\documentclass[10pt,reqno]{amsart}
\usepackage{graphicx}
\usepackage{amsmath,amsthm}
\usepackage{amsfonts}
\usepackage{amscd}
\usepackage[all]{xy}
\usepackage{amssymb}
\usepackage[a4paper]{geometry}
\usepackage{subfigure}
\usepackage{mathtools}
\usepackage{upgreek}
\usepackage{enumitem}
\usepackage{tgpagella}%Font
\usepackage{comment}

\usepackage{hyperref}

\hypersetup{
	bookmarks=true,         % show bookmarks bar?
	unicode=false,          % non-Latin characters in Acrobat?s bookmarks
	pdftoolbar=true,        % show Acrobat?s toolbar?
	pdfmenubar=true,        % show Acrobat?s menu?
	pdffitwindow=false,     % window fit to page when opened
	pdfstartview={FitH},    % fits the width of the page to the window
	pdftitle={Random Products of Standard Maps},    % title
	pdfauthor={Pablo D. Carrasco},     % author
	pdfkeywords={Standard Map, Froeschl\'e family, Non-uniform Hyperbolicity, Partially Hyperbolic Diffeomorphisms, Physical measures, Arnold diffusion.}, % list of keywords
	pdfnewwindow=true,      % links in new PDF window
	colorlinks=true,       % false: boxed links; true: colored links
	linkcolor=blue,          % color of internal links (change box color with linkbordercolor)
	citecolor=green,        % color of links to bibliography
	filecolor=magenta,      % color of file links
	urlcolor=cyan           % color of external links
}

\newtheorem{theorem}{Theorem}[section]

\newtheorem{proposition}{Proposition}[section]
\newtheorem{lemma}{Lemma}[section]
\newtheorem{corollary}{Corollary}[section]
\newtheorem{definition}{Definition}[section]
\newtheorem{remark}{Remark}[section]

\newtheorem*{main}{Main Theorem}

\newtheorem*{coroA}{Corollary A}
\newtheorem*{coroB}{Corollary B}

\newtheorem*{theoremsn}{Theorem}

\newcommand{\norm}[1]{\left\Vert#1\right\Vert}
\newcommand{\normm}[1]{{\left\vert\kern-0.25ex\left\vert\kern-0.25ex\left\vert #1 
    \right\vert\kern-0.25ex\right\vert\kern-0.25ex\right\vert}}

\newcommand{\Real}{\mathbb R}
\newcommand{\ep}{\epsilon}

\newcommand{\al}{\alpha}
\newcommand{\be}{\beta}

\newcommand{\ga}{\gamma}
\newcommand{\To}{\longrightarrow}

\newcommand{\Tor}{\mathbb{T}^2}
\newcommand{\Toro}{\mathbb{T}}
\newcommand{\oo}{\infty}
\newcommand{\esp}{\vspace{0.2cm}}

\newcommand{\Fu}{\ensuremath{\mathcal{W}^u}}

\newcommand{\Futf}{\ensuremath{\mathcal{W}^u_{\wtf}}}

\newcommand{\F}{\ensuremath{\mathcal{W}}}

\newcommand{\Wu}[1]{\ensuremath{W^u(#1)}}
\newcommand{\Wutf}[1]{\ensuremath{W^u_{\wtf}(#1)}}

\newcommand{\wtf}{\widetilde{f}}

\newcommand{\Z}{\mathbb{Z}}

\begin{document}
\author{Pablo D. Carrasco}
\address{ICEx-UFMG, Avda.\@ Presidente Ant\^onio Carlos 6627, Belo Horizonte-MG, BR31270-901}
\email{pdcarrasco@mat.ufmg.br}
%\thanks{The author was partially supported by FAPESP Project\# 2014/119262-0.}

\title{Random products of Standard maps}


\date{\today}

\begin{abstract}
	We develop a general geometric method to establish the existence of positive Lyapunov exponents for a class of skew products. The technique is applied to show non-uniform hyperbolicity of some conservative partially hyperbolic diffeomorphisms having as center dynamics coupled products of standard maps, notably for skew-products whose fiber dynamics is given by (a continuum of parameters in) the Froeschl\'e family. These types of coupled systems appear as some induced maps in models for the study of Arnold diffusion. 
	
	Consequently, we are able to present new examples of partially hyperbolic diffeomorphisms having rich high dimensional center dynamics. The methods are also suitable for studying cocycles over shift spaces, and do not demand any low dimensionality condition on the fiber. 
\end{abstract}

\maketitle
\tableofcontents

\section{Introduction}

Let $f$ be a $\mathcal{C}^2$ diffeomorphism of a compact Riemannian manifold $M$ preserving a smooth probability measure $\mu$. A central tool to detect  chaotic behavior in the system is by means of its Lyapunov exponents. 

\begin{definition}
	For $p\in M, v\in T_pM\setminus\{0\}$, the Lyapunov exponent of $v$ is 
	\begin{equation}\label{Lyapunov}
	\chi(p,v)=\limsup_{n\rightarrow+\oo}\frac{\log\norm{d_pf^n(v)}}{n}.
	\end{equation}
\end{definition}

It is a consequence of Oseledet's Multiplicative Ergodic Theorem \cite{Oseledets} that the above limit exists for $\mu$ almost every $p\in M$ and every  $v\in T_pM\setminus\{0\}$. The existence of non zero Lyapunov exponents on a set of positive $\mu$ measure guarantees abundance of exponentially diverging orbits, either for the future or the past. Of particular importance is the case when all exponents are different from zero.

\begin{definition}
	The system $(f,\mu)$ is non-uniformly hyperbolic (NUH) if its Lyapunov exponents are non-zero $\mu$-a.e., in other words for $\mu$ almost every $p\in M$ it holds
	\[
	\forall v\in T_pM\setminus\{0\},\quad \lim_{n\To+\oo}\frac{\log\norm{d_pf^n(v)}}{n}\neq 0.
	\]
\end{definition}

The concept of non-uniformly hyperbolicity was introduced by Y. Pesin as a generalization of the classical hyperbolic diffeomorphisms, and allow for much more flexibility than its uniform counterpart. However, notwithstanding the large industry dedicated to their study, NUH systems are still not very well understood, mainly because establishing the existence of non zero exponents requires an asymptotic analysis of almost every orbit. This difficulty is reflected in the very limited available pool of examples of truly (i.e.\@ non Anosov) NUH diffeomorphisms.

The obstacle appears already for conservative surface maps, where showing non-uniform hyperbolicity is reduced (since the sum of exponents for a conservative map is zero \cite{LedraExp}) to establishing positivity of a single exponent. Perhaps the most famous example where non-uniform hyperbolicity is unknown is given by Chirikov-Taylor standard family (cf. \cite{Chirikov2008}): for $r>0$, $s_r:\mathbb{T}^2=\Real^2/(2\pi\mathbb{Z})^2\rightarrow\mathbb{T}^2$ is the diffeomorphism given by
\begin{equation}\label{standardeq}
s_r(x,y)=\begin{pmatrix}
2 & -1\\
1 & 0
\end{pmatrix}\cdot \begin{pmatrix}
x\\
y
\end{pmatrix}+r\begin{pmatrix}
\sin(x)\\
0
\end{pmatrix}.
\end{equation}

Interacting pairs of these maps were introduced to study Arnold diffusion for systems of coupled oscillators \cite{chirikov1979}, albeit the original research was mainly numerical. Each $s_r$ preserves the Lebesgue measure, and it is a major open problem in smooth ergodic theory to show the existence of non-zero Lyapunov exponents. Such behavior is evidenced numerically (cf. \cite{Chirikov2008}), but currently NUH of $s_r$ is still unknown for even a single parameter \cite{Pesin2010}.  It is know however that the dynamics of  $s_r$ is very complicated. See for example \cite{Plenty,stochasticsea} and references therein.

In \cite{NUHD} P. Berger and the author proposed to study a random version of the standard map. We considered a linear hyperbolic automorphism of $A:\mathbb{T}^2\rightarrow \mathbb{T}^2$ and studied the diffeomorphism\footnote{Here $[r]$ denotes the integer part of $r$.}
\begin{equation}\label{BerCar}
f_r(x,y,z,w)=(A^{[2r]}\cdot(x,y),s_r(z,w)+P\circ A^{[r]}\cdot(x,y))\quad P(x,y)=(x,0).
\end{equation}
The coupled map $f_r$ can be seen as a random perturbation of the dynamics of $s_r$: computing its derivative one verifies that the action $df_r|\{0\}\times\Real^2$ coincides with $ds_r|\{0\}\times\Real^2$, so one can think $f_r$ as a family of standard maps driven by the random motion determined by the base dynamics. This point of view (studying random versions of maps) is well established, and related to the study of classical fast-slow systems of differential equations. For considerations on the statistical properties of similar maps the reader can check for example \cite{Dolgoaverage,Korepanov2017,limitfastslow}.

In that article we established the following.

\begin{theoremsn}
	There exists $r_0>0$ such that for $r\geq r_0$ one can find a $\mathcal{C}^2$ neighborhood $\mathcal{U}_r$ of $f_r$ and $C(r)>0$ such that 
	\[
	g\in \mathcal{U}_r\text{ is conservative }\Rightarrow \text{for Lebesgue}-a.e.(p),\forall v\in T_p\Toro^4\setminus\{0\},\ |\chi(p,v)|>C(r).
	\]
	Thus every such $g$ (in particular $f_r$) is NUH. The constant $C(r)$ goes to infinity as $r\To \oo$. 
\end{theoremsn}

The goal of this paper is to refine some of the ideas and techniques introduced in \cite{NUHD} and establish positivity of center exponents for a larger class of examples. In the referred article most of the arguments were subordinate to a specific case, as the main motivation of the authors was to study a concrete center behavior (given by the standard family), whereas here we are more interested in developing general results with ample applicability, particularly without the low dimensionality restriction in the fiber. For unfolding broad methods, we fine-tune the central notions of \emph{admissible curves} and \emph{adapted fields} introduced in \cite{NUHD}, providing a more abstract definition readily suitable for dealing with more general systems. 

Our theorems will be stated in large generality, hence it seems appropriate to give an example of their consequences.  The reader is referred to the third Section where more general versions of this and other examples are discussed thoroughly, and where a comparison between the present methods and others in the literature is also given.

Consider an hyperbolic matrix $A\in Sl(2,\mathbb{Z})$, $e\geq 2$ a natural number and $r>0$; define $g_r:\Tor\times\Toro^{2e}\rightarrow \Tor\times\Toro^{2e}$ by the formula
\begin{align}\label{standardproduct}
g_r(z_0,z_1,\cdots, z_e)=\Big(A^{[2r]}(z_0),s_r(z_1)+P\circ A^{[r]}(z_0),\cdots,
s_r(z_e)+P\circ A^{[r]}(z_0)\Big) 
\end{align}
where $\displaystyle{z_i=(x_i,y_i)\in\Tor}$. Then $g_r$ is a generalized version of the map $f_r$ of \cite{NUHD}, having $e$ coupled standard maps acting on the fiber instead of just one. Note in particular that if $V=\{0\}\times \Real^{2e}$ then $dg_r|V=ds_r\times\cdots \times ds_r$, the product taken $e$ times.

As consequence of the results in this article we get the following.

\begin{theorem}\label{teo.introcoupled}
	There exists $r_0>0$ such that for $r\geq r_0$ one can find a $\mathcal{C}^2$ neighborhood $\mathcal{U}_r$ of $g_r$ such that 
	\[
	g\in \mathcal{U}_r\text{ is conservative }\Rightarrow \text{for Lebesgue}-a.e.(p),\forall v\in T_p\Toro^{2(e+1)}\setminus\{0\},\ |\chi(p,v)|>\frac{3}{5}\log r.
	\]
\end{theorem}

It is worth to point out that other available methods in the literature seem to be unfitted to study this type of map, and the above result is, at least as known to the author, the only formal proof of the existence of non-zero Lyapunov exponents for this higher dimensional version of coupled standard maps. It is also interesting to remark that the map $g_{r}$ is partially hyperbolic (see the next section for the definition), hence in passing we provide a concrete example of a partially hyperbolic system with interesting higher dimensional center behavior.

A particular instance of coupled standard maps is given by the  family of diffeomorphisms $u_{r,\tau}:\mathbb{T}^2\times\mathbb{T}^2\rightarrow\mathbb{T}^2\times\mathbb{T}^2$,
\begin{align}\label{coupledstandard}
u_{r,\tau}(x,y,z,w)=(2x-y+r\sin(x)+\tau\sin(x+z),x,2z-w+r\sin(z)+\tau\sin(x+z),z)
\end{align}
where $0\leq \tau<1$ above is a small parameter. Similar types of systems appeared in the study of Arnold diffusion; in the specific case of $u_{r,\tau}$ above, an equivalent map was studied by B. P. Wood,  A. J. Lichtenberg and M. A. Lieberman in \cite{Wood1990} where they provided numerical evidence of Arnold diffusion between the stochastic layers, for small $\tau$. See also the Froeschl\'e family example in the third Section. It is apparent that these systems are at least as complex, if not more, than $s_r$. From the previous theorem we get.

\begin{corollary}\label{cor.Frosty}
	Define the family of maps $g_{r,\tau}:\mathbb{T}^6\rightarrow \mathbb{T}^6$ by 	
	\begin{align*}
	&g_{r,\tau}(x_1,\ldots,x_6):=(A^{[2r]}(x_1,x_2),u_{r,\tau}(x_3,\ldots,x_6)+\varphi_r(x_1,x_2))\\
	&\varphi_r(x_1,x_2):=(P\circ A^{[r]}(x_1,x_2),P\circ A^{[r]}(x_1,x_2)).
	\end{align*}
	where $A\in SL(2,\mathbb{Z})$ is hyperbolic. Then there exists $r_{0}>0$ such that for $r\geq r_0$ one can find $\tau_0(r)>0$  such that for $0\leq\tau\leq\tau_0(r)$ the map $g_{r,\tau}$ is $\mathcal{C}^2$ - robustly NUH, meaning that any volume preserving map which is sufficiently $\mathcal{C}^2$ - close to $g_{r,\tau}$ is also NUH.
\end{corollary}

\section{Statement of the Main Result}
\subsection{Products of conservative systems with some hyperbolicity}

A difficult and not well understood problem in smooth ergodic theory is establishing the existence of non zero exponents for systems that are hyperbolic on a large but not invariant set. One of the main complications is that in visits to the complement of the hyperbolic set, vectors that were previously expanded can 
be sent to one of the contracted directions, thus loosing expansion. This complication is particularly present in the conservative setting, since typically the complement of the hyperbolic set has positive measure, and thus is recurrent. The Chirikov-Taylor standard map family \eqref{standardeq} falls into this category; for a higher dimensional example one can simply take the product of several standard maps. As we mentioned before, no parameter such that any of these systems is NUH is currently known.

In this article we will be considering a random version of these products. In pursuit of generality and to free ourselves from the specific formula of the map, we will enumerate the properties required on each of the factors for our method to work. Nevertheless, we suggest the reader to keep in mind the specific example given in \eqref{standardproduct}.

\esp

\noindent\textbf{Notation:}

\begin{itemize}[leftmargin=*]
	\item[-] Consider $V, W\subset \Real^d$ non-trivial sub-spaces such that $\Real^d=V\oplus W$. The cone of size $\alpha>0$ centered around $V$ in $V\oplus W$ is the set
	\[
	\Delta_{\alpha}:=\{(v,w)\in V\oplus W: \norm{w}<\alpha\norm{v}\}\cup\{(0,0)\}.
	\]
	For such a set, $\overline{\Delta}_{\alpha}$ denotes its closure in $\Real^d$, and $\mathrm{C}\Delta_{\alpha}=\Real^{d}\setminus\Delta_{\alpha}$ is the complementary cone of $\Delta_{\alpha}$. We will mostly consider cones centered around $\Real\times\{0\}$ in $\Real\times\{0\}\oplus \{0\}\times\Real$; note that in this particular case such a cone is of the form
	\[
	\Delta_{\alpha}:=\Real\cdot\{(1,w)\in\Real^2:|w|<\alpha\}
	\]
	whereas its complimentary cone is 
	\[
	\mathrm{C}\Delta_{\alpha}:=\Real\cdot\{(1,w)\in\Real^2:|w|\geq \alpha\}\cup\Real\cdot(0,1).
	\]
	
	We will need to consider also finite decompositions of $\mathrm{C}\Delta(\alpha)$ into sub-cones. For concreteness, let 
	\begin{align*}
	\mathrm{C}\Delta^+_{\al}&=\Real\cdot\{(1,w)\in\Real^2:w\geq \alpha\}\cup \Real\cdot(0,1)\\
	\mathrm{C}\Delta^-_{\al}&=\Real\cdot\{(1,w)\in\Real^2: w\leq -\alpha\}\cup \Real\cdot(0,1).
	\end{align*}
	Then $\mathrm{C}\Delta_{\al}=\mathrm{C}\Delta^+_{\al}\cup\mathrm{C}\Delta^+_{\al}$ and $\mathrm{C}\Delta^+_{\al}\cap \mathrm{C}\Delta^-_{\al}=\Real\cdot(0,1)$. These notations extend naturally to cone fields on $\Tor$.
	
	\item[-] If $T:(V,\norm{\cdot}_V)\rightarrow (W,\norm{\cdot}_W)$ is a linear map between normed vector spaces, we denote its operator norm and conorm by
	\begin{align*}
	\norm{T}:=\sup\{\norm{T(v)}_W:\norm{v}_V=1\}\\
	m(T):=\inf\{\norm{T(v)}_W:\norm{v}_V=1\}
	\end{align*}
	In case that $T$ is invertible, it holds $\norm{T^{-1}}=\frac{1}{m(T)}$.
	
	\item[-] Let $N$ be a compact manifold and $\pi:E\rightarrow N$ a continuous vector bundle equipped with a Riemannian metric $\norm{\cdot}$. If $T:E\rightarrow E$ is a bundle map we will write
	\begin{align*}
	\norm{T}&:=\max_{p\in N}\{\norm{T_p}:T_p:E_p\to E_{t(p)}\}
	\end{align*}
	where $t$ is the induced map by $T$ on $N$.  If $T$ is an automorphism, it holds \[\norm{T^{-1}}=\frac{1}{\min_{p\in N}\{m(T_p)\}}.\]
	
	\item[-] A subset of the form $A=[a,b]\times\Toro\subset \Tor$ or $A=\Toro\times[a,b]\subset\Tor$ is called a band of base $[a,b]$, and we say that it has length $l(A)=b-a$. As a mild abuse of language, we also call a band to the union of finitely many sets as before, provided that the bases are in the same coordinate of $\Tor$. 
\end{itemize}

\esp

To express quantitatively the required conditions we find convenient to work in the context of parametrized families. Consider a family of conservative diffeomorphisms $\{S_r:\Tor\rightarrow\Tor\}_r$ satisfying the following conditions.

\begin{enumerate}
	\item[i)] The existence of a continuous cone field $\Delta_{r}=\{\Delta_r(y)=\Delta_{\al_r,r}(y)\subset\Real^2\}_{y\in\Tor}$ on $\Tor$. Here we are identifying for $y\in\Toro^2, T_y\Tor=\Real^2$.

	\item[ii)]  The existence of bands $\mathcal{C}_{r},\mathcal{B}_r\subset \Tor$ whose bases are in the same coordinate of $\Tor$. The set $\mathcal{C}_{r}$ will be referred as the \emph{critical region} of $S_{r}$. To simplify the exposition and with no loss of generality, we will assume that the bases of these bands are in the first coordinate. 
	
\end{enumerate}
We denote by $\beta(r)$ the smallest expansion for vectors in $\Delta_{r}$ at points outside the critical region, and by $\zeta(r)$ the smallest contraction at points in $\Toro^2$ , i.e.\ 

\begin{align}
\beta(r)&:=\inf\{m\big(d_yS_{r}|\Delta_{r}(y)\big):y\not\in\mathcal{C}_{r}\}\\
%\eta(r)&:=\inf\{m(d_yS_{r}):y\in\mathcal{B}_{r}^+\cup \mathcal{B}_{r}^-\}\\
\zeta(r)&:=\inf\{m(d_yS_{r}):y\in\Tor\}(=\frac{1}{\norm{dS^{-1}}}).	
\end{align}
\noindent\textbf{Hypotheses on $S_r$}: there exist $\sigma\in \mathbb{N},0<R<2\pi$ such that for $r$ large the following conditions are verified.
\begin{enumerate}
	\item[\textbf{S-1}] It holds 
	\begin{itemize}
		\item the length of the critical region converges to zero as $r$ goes to infinity,
		\[
		l(\mathcal{C}_{r})\xrightarrow[r\To \oo]{}0.
		\]
		\item Associated to $\mathrm{C}\Delta^+_{r},\mathrm{C}\Delta^+_{r}$ in the decomposition of $\mathrm{C}\Delta_{r}$ there exist bands $\mathcal{B}_r^+,\mathcal{B}_r^-\subset \mathcal{B}_r$ such that $\mathcal{B}_r=\mathcal{B}_r^+\cup\mathcal{B}_r^-$ and furthermore
		\[
		\min\{l(\mathcal{B}_{r}^+),l(\mathcal{B}_{r}^-)\}\geq R.
		\]
		\item The size of $\Delta_r$ is bounded from below as a function of $r$,
		\[
		\inf_r \min\{\al_r(p):p\in\Tor\}>0.
		\]
		\item Vectors in $\Delta_{r}$ at points outside the critical region are uniformly expanded
		\[
		\beta(r)> 1\geq \zeta(r)
		\]
		\item The maximal contraction of $S_r$ is controlled by the expansion in the cone field $\Delta_{r}$ by the following relation
		\[
		\beta(r)^6\zeta(r)^{1/\sigma}>1.
		\]
	\end{itemize}	
	\item[\textbf{S-2}] 
	\begin{itemize}
		\item Invariance of $\Delta_r$: if $y\not\in\mathcal{C}_{r}$ then $v\in    \overline{\Delta}_{r}(y)\Rightarrow d_yS_{r}(v)\in \Delta_{r}(S_{r}(y))$. 
		\item Restitution of the expansion direction: it holds
		\begin{itemize}
			\item $y\in\mathcal{B}_r^{+}, v\in\mathrm{C}\Delta_r^+(y)$ then $d_yS_{r}(v)\in \Delta_{r}(S_{r}(y))$.
			
			\item $y\in\mathcal{B}_r^{-}, v\in\mathrm{C}\Delta_r^-(y)$ then $d_yS_{r}(v)\in \Delta_{r}(S_{r}(y))$.	
		\end{itemize}	
		It is further assumed that the angle between $d_yS_r(v)$ and the boundary of $\Delta_{r}(S_{r}(y))$ is bounded away from zero uniformly in $r$.
	\end{itemize}
\end{enumerate}

\begin{remark}\label{general}\hfill  
	\begin{enumerate}
		\item  More generally, we could consider a finite decomposition into sub-cones $\mathrm{C}\Delta=\mathrm{C}\Delta^1\cup\cdots\cup \mathrm{C}\Delta^k$. In this case we require the existence of associated bands $\mathcal{B}_r^1,\cdots,\mathcal{B}_r^k$ satisfying analogous properties as the listed above.
		
		\item The techniques of this article also permit to consider families $\{S_r:\Toro^d\rightarrow\Toro^d\}_r$ with $d>2$ satisfying similar conditions. In this case, the bands are replaced by (finite union of) sets of the form  $\Toro^{k-1}\times[a,b]\times\Toro^{d-k}$, and there are some dimension restrictions that have to be imposed on the cones (if $\Delta_{r}$ is a cone centered around $\Real^k$ then $k\geq \frac{d}{2}$ in order to satisfy the last part of $\textbf{S-2}$). Since this complicates the notation, and no compelling example of this situation seems available, the author decided to consider the case $d=2$ and leave the problem of the general case to the interested reader. See also the comments after the proof of Proposition \ref{ergog}.	
	\end{enumerate}		
\end{remark}	

\esp

It is worth to emphasize that the conditions postulated on $S_r$ do not guarantee positivity of any of its exponents. Note also that the quantitative requirement in the expansion hypothesis (the last part of \textbf{S-1}) is very weak.

\esp

We will be considering the product  $S_r=S_{1,r}\times\cdots\times S_{e,r}:\Toro^{2e}\rightarrow\Toro^{2e}$ of $e$ factors $S_{i,r}:\Tor\rightarrow\Tor$ satisfying $\textbf{S-1},\textbf{S-2}$ above: as an mild abuse of language we say in this case that $\{S_r\}_r$ satisfies $\textbf{S-1},\textbf{S-2}$.  We will identify $\Tor_i=\{0\}\times\cdots\times\Tor\times\cdots\times\{0\}\subset\Toro^d$ (in the $i$-th position), and likewise $\Real^{2}_i=\{0\}\times\cdots\times\Real^{2}\times\cdots\times\{0\}\subset\Real^{2e}=T_y\Toro^{2e}, \forall y\in\Toro^{2e}$. In the same way, objects associated to the map $S_{i,r}$ will be denoted by the corresponding subscript (for example, $\mathcal{C}_{i,r}$ denotes the critical region of $S_{i,r}$).

\subsection{Partially hyperbolic skew products}

As we mentioned in the Introduction there are very few available methods to establish non uniform hyperbolicity. It is not surprising then that a fair amount of the literature imposes some extra hypothesis on the map $f$ considered, customarily the existence of some uniformly hyperbolic directions. A case of special interest is when the diffeomorphism $f$ is partially hyperbolic. We recall the definition below.

\begin{definition}
	The diffeomorphism $f:M\rightarrow M$ is weakly partially hyperbolic (wPH) if there exists a continuous $df$-invariant splitting $TM=E^{cs}\oplus E^u$ (i.e. $d_pf(E_p^u)=E_{f(p)}^u,d_pf(E_p^{cs})=E_{f(p)}^{cs}\ \forall p\in M$) and constants $C>0,\lambda>1,0<K<1$ such that for every $p\in M$, for every unit vectors $v\in E^u_p,w\in E_p^{cs}$ and for all $n\geq 0$ it holds
	\begin{enumerate}
		\item $\norm{d_pf^n(v)}\geq C\lambda^n\ $   (uniform expansion in $E^u$).
		\item $\norm{d_pf^n(w)}\leq K\cdot \norm{d_pf^n(v)}\ $  (domination between $E^{cs}$ and $E^u$).
	\end{enumerate}
	$f$ is partially hyperbolic (PH) if both $f$ and $f^{-1}$ are wPH: in this case there exists a $df$-invariant splitting $TM=E^s\oplus E^c\oplus E^u$ such that vectors in $E^u$ (resp. in\@ $E^s$) are exponentially expanded (resp.\@ contracted). The bundles $E^s,E^u,E^c$ are denominated the stable, unstable and center bundle of $f$.
\end{definition}

See \cite{PesinLect,Beyond,survey3d} for further information on Partially Hyperbolicity. We content ourselves reminding the reader that (weakly) partial hyperbolicity is a
$\mathcal{C}^1$ open condition.

If $f$ is PH, the Lyapunov exponents corresponding to $E^s\oplus E^u$ are non-zero, so it is enough to study the Lyapunov exponents for vectors in $E^c$ (these will be referred as \emph{center exponents}). It is fair to say that so far the main focus of research in this topic has been the case when the center exponents have a definite sign (all positive or negative), or when there exist dominated splitting among the subspaces corresponding to Lyapunov exponents of opposite signs\footnote{If $\Lambda\subset M$ is $f$-invariant, we say that a $df$-invariant sum $E\oplus F\subset T_{\Lambda}M$ is a dominated splitting if there exists $0<K<1$ such that for $n\geq 0$, $p\in \Lambda$ and unit vectors $v\in E(p),w\in F(p)$ it holds $\norm{d_pf^n(v)}\leq K\cdot \norm{d_pf^n(w)}$.}. See for example \cite{SRBMostexp,Mostcont,StaErgNeg,PartHypLya}.

The map given in \cite{NUHD} is also PH, although of a different type. Its center bundle is two dimensional, with corresponding exponents of opposite signs, whereas it does not admit a dominated decomposition into one-dimensional sub-bundles, and even more, it is $\mathcal{C}^2$ - \emph{robustly NUH}. In the non-conservative case, results of the same type were obtained before by M. Viana \cite{Multinonhyp}.

A popular example of PH diffeomorphism are the so called \emph{PH skew products}. Consider $A:N\rightarrow N$ an Anosov diffeomorphism and let $S:\mathbb{T}^d\rightarrow \mathbb{T}^d$ be a differentiable map. Given a smooth function $\varphi:N\rightarrow \mathbb{T}^d$, the skew product of $A$ and $S$ with respect to $\varphi$ is the (bundle) map $f:N\times \mathbb{T}^d\rightarrow N\times \mathbb{T}^d$ given as
\[
f(x,y)=(A(x),S(y)+\varphi(x)).
\]
We denote $f=A\times_{\varphi} S$ and call $A$ the \emph{base map}, $S$ the \emph{fiber map} and $\varphi$ the \emph{correlation map}. The manifold $M:=N\times \mathbb{T}^d$ is equipped with the product of any pair of metrics in $N,\Toro^d$.

Since $A$ is hyperbolic, natural domination conditions between $dA$ and $dS+d\varphi$ imply that $f$ is PH, i.e.\@ a \emph{PH skew product}, with center bundle
\begin{equation}\label{V}
V= \{0\}\times T\Toro^d.	
\end{equation}
If $TN=E^s_A\oplus E^u_A$ is the hyperbolic decomposition corresponding to $A$, we extend these bundles to $N\times \mathbb{T}^d$ using the same nomenclature (i.e., $E^s_A=E^s_A\times\{0\},E^u_A=E^u_A\times\{0\}$).  

Note that the inverse of a skew product as defined before is not necessarily a skew-product, but rather of the form
\[
f^{-1}(x,y)=(A^{-1}(x),\psi(x,y)).
\]
These types of systems are sometimes called \emph{fibered}.

For establishing positivity of center exponents of (families of) skew products we require some control in the dynamics along unstable directions of the base map. We will make the following hypotheses.

\esp

\noindent\textbf{Standing hypotheses for the rest of the article:} The base maps of the skew products are linear automorphisms of the 2-torus, i.e.\@ $N=\mathbb{T}^2$ and $A\in SL(2,\mathbb{Z})$. The correlation functions $\varphi:\Toro^2\to\Toro^{2e}$ are linear\footnote{Meaning that they are induced by linear maps $\Real^2\to\Real^{2e}$.}.

\esp

The linearity condition on the correlation map is not central, and is used to simplify further requirements. On the other hand, our methods use strongly conformality of the action of $A$ on its unstable directions. The above hypotheses are strong,  but even this case is not currently well understood and falls out the category of examples discussed in the literature (cf.\@ third section); besides, we are more interested in the behavior of the map in the fiber directions. We will thus study families of PH skew products, where the fiber maps are given by a family $\{S_r=S_{r,1}\cdots S_{r,e}:\Toro^{2e}\rightarrow\Toro^{2e}\}_r$  satisfying \textbf{S-1},\textbf{S-2}, and where the correlation functions determine a weak but non vacuous interplay between the strong unstable directions coming from the base dynamics and the fiber directions.

\begin{remark}\label{remproducto}
	The simplest case is when $e=1$ i.e.\@ there exists a single cone defined in the complement of its critical region where vectors are expanded under the action of the derivative. An example of this situation is given by the (uncorrelated) product $f_r:\Toro^4\rightarrow\Toro^4$ with $f_r=A\times s_r$, where $A\in SL(2,\mathbb{Z})$ is hyperbolic. Our techniques do no apply to these kind of map, although related cases are treated in Corollary A (cf. the next Section). It is known that $\mathcal{C}^2$ conservative perturbations of $f_r$ above are NUH, provided that $r$ is small \cite{Marin2016}.
\end{remark}

\subsection{Coupled families}

In this part we state explicitly the conditions required on the correlation functions for the techniques in this article. If the reader so prefers, she/he can focus in the specific correlation function used in \eqref{standardproduct}, where all conditions are more transparent to check.

Let $\{f_r=A_r\times_{\varphi_r} S_r\}_r$ be a family of skew products where  $\{S_r=S_{r,1}\times\cdots \times S_{r,e}:\Toro^{2e}\rightarrow\Toro^{2e}\}_r$ satisfies \textbf{S-1},\textbf{S-2}, and $A_r\in SL(2,\mathbb{Z})$ are uniformly hyperbolic. Associated to $A_r$ we have a decomposition into eigenspaces  $\Real^2=E^u_{A_r}\oplus E^s_{A_{r}}$, and we choose unit vectors $e^u_{A_r},e^s_{A_r}$ generating respectively $E^u_{A_r}, E^s_{A_r}$ with $A_r(e^u_{A_r})=\lambda_r\cdot e^u_{A_r}, A_r(e^u_{A_r})=\tau_r\cdot e^s_{A_r}$ where $1<\lambda_r=\frac{1}{\tau_r}$. Note that since we are assuming that $\varphi_r$ is linear, $\norm{\varphi_r|E^u_{A_r}}=\norm{\varphi_r(e^u_{A_r})}$ and $\norm{\varphi_r|E^s_{A_r}}=\norm{\varphi_r(e^s_{A_r})}$.

For $1\leq i\leq e$ let  $P_{i}:\Real^{2e}\rightarrow \Real^2_i$ be the projection onto the first coordinate of $\Real^2_i$, namely
\begin{equation}
P_i(x_1,\ldots,x_{2e})=(x_{2i-1},0)
\end{equation}

\begin{definition}
	We say that the coupling in $\{f_r\}_r$ is adapted if
	\begin{enumerate}[leftmargin=*]
		\item[\textbf{A-1}]
		\begin{itemize}
			\item $\displaystyle{\frac{\norm{\varphi_r|E^s_{A_r}}+\norm{dS_r}^3}{\norm{\varphi_r|E^u_{A_r}}}\xrightarrow[r\To+\oo]{}0, \frac{\norm{\varphi_r}}{\lambda_r}\xrightarrow[r\To+\oo]{}0}$, $\displaystyle{ \norm{\varphi_r|E^s_{A_r}}\cdot\norm{dS^{-1}}\xrightarrow[r\To+\oo]{}0.}$\\
			
			\item There exists $l\in\mathbb{N}$ such that 
			\[
			\frac{\lambda_r}{\norm{\varphi_r}^{l}}\xrightarrow[r\To+\oo]{}0,\quad \frac{\norm{dS^{-1}_r}^{3l}(\norm{dS_r}^{3l}+\norm{d^2S_r}^{3l})}{\lambda_r}\xrightarrow[r\To+\oo]{}0
			\]\\	
		\end{itemize}
		\item[\textbf{A-2}] 
		\begin{itemize}
			
			\item $\displaystyle{\min_{1\leq i\leq e}\norm{P_i\circ\varphi_r|E^u_{A_r}}>0}$.
			
			\item $\displaystyle{\max_{1\leq\i\leq e}\frac{\norm{P_i\circ\varphi_r|E^s_{A_r}}+\norm{P_i\circ dS_r}}{\norm{P_i\circ \varphi_r|E^u_{A_r}}}\xrightarrow[r\To+\oo]{}0}$.
			
			\item $\displaystyle{K(r):=\frac{\min_{1\leq i\leq e}\norm{P_i\circ \varphi_r|E^u_{A_r}}}{\max_{1\leq i\leq e} \norm{P_i\circ \varphi_r|E^u_{A_r}}}\xrightarrow[r\To+\oo]{} 1.}$
		\end{itemize}
	\end{enumerate}
\end{definition}

\begin{remark} Above, it is tacitly assumed that $\norm{dS_r}\geq 1$ (hence $\norm{dS^{-1}_r}\geq 1$ since $S_r$ preserves volume).
	
\end{remark}	

Condition \textbf{A-1} simply establishes domination among the relevant quantities. It implies that for large $r$ the map $f_r$ is PH (Corollary \ref{partialhyp}), and moreover the angles $\angle(E^u_{f_r},E^u_{A_r}),\angle(E^u_{f_r},E^s_{A_r})$ converge to zero as $r\To+\oo$ (cf. Lemma \ref{converfibf}). In practice, the norms of $dS_r,dS_r^{-1},d^2S_r$ will behave polynomially in $r$ while $\lambda_r$ is exponential, so \textbf{A-1} will be simple to check.  As for condition \textbf{A-2}, it means that $\varphi_r$ provides a interaction of the unstable directions of $A_r$ with every $S_{i,r}$ and all these interactions are comparable and close to conformal (for some suitable metric on the center bundle of $f_r$).

\begin{remark}
	The choice of picking the first coordinate in the definition of $P_i$ is determined by the form of the critical region of $S_{i,r}$. In the case where $\mathcal{C}_{i,r}$ is a band in the other coordinate  we need to modify $P_i$ accordingly. The reader will find no difficulty in adapting the results to this  situation.   
\end{remark}

\begin{main}\hypertarget{mainteo}
	Consider the family $\{f_r=A_r\times_{\varphi_r} S_r\}_{r}$ with $A_r\in SL(2,\mathbb{Z})$ hyperbolic, $\{S_r\}_r$ satisfying conditions \textbf{S-1}, \textbf{S-2}, and where the coupling is adapted. Then there exists $r_0$ such that for every $r\geq r_0$ the map $f_r$ is PH and furthermore there exists $Q(r)>0$ and a full Lebesgue measure set $NUH_r\subset M$ satisfying: for every $p\in NUH_r$ there exists $U(p)\subset E^c_{f}(p)$ sub-space of dimension greater than equal to $e$, and so that 
	\[
	v\in U(p)\setminus\{0\}\Rightarrow \chi(p,v)>Q(r). 
	\]
	The same is true for any conservative map in a $\mathcal{C}^2$ neighborhood $\mathcal{U}_r$ of $f_r$.
\end{main}

\esp

\begin{remark}\label{bound}
	One can estimate $Q(r)\geq \min_{1\leq i\leq e}\frac{0.99\sigma}{\sigma+1}\log(\beta_i(r)^6\zeta_i(r)^{1/\sigma})$. cf. the proof of Proposition \ref{fonda} in page \pageref{prueba} and \ref{ergog}. 	
\end{remark}

Let us point out to the reader that in the setting that we are studying, due to Oseledet's theorem there exist for $\mu$ almost every point $p\in M$ a natural number $1\leq k(p)\leq \dim M$, real numbers $\chi_1(p),\ldots,\chi_k(p)$ and subspaces $E^1(p),\cdots, E^k(p)\subset T_pM$ such that $0\neq v\in E^i(p)\Rightarrow \chi(p,m)=\chi_i(p)$. We call $\dim E^i(p)$ the \emph{multiplicity} of the exponent $\chi_i$ at $p$. Our Main Theorem asserts that under its hypotheses, for large $r$ it holds that for $\mu$ almost every point $p$ the sum of the  multiplicities corresponding to center exponents at $p$ that are larger than $Q(r)$ is at least $e$.  

\esp

The rest of the article is organized as follows. In the next section we show how to apply the stated results to the important case of random surface maps, and we discuss the existence of physical measures for such systems. We also consider several concrete examples, in particular (more general versions of) the random coupled standard maps system. This section also discusses how to apply the techniques to other classical realization of random maps when we interchange the base dynamics by a shift map. Then in the fourth section we present our main tool: admissible curves and adapted vector fields. These are used in the following section to establish the Main Theorem. We finish the article with a short appendix where some side technical considerations are examined.

\section{Applications and examples}

Our Main Theorem can be applied to establish non uniform hyperbolicity of some coupled families $\{f_r=A_r\times_{\varphi_r} S_r\}_{r}$. In this Section give examples and derive some consequences.

\subsection{Existence of physical measures and surface maps}

We recall that the basin of of attraction of a $f$ - invariant measure $\mu$ is

\[
B(\mu)=\{p\in M:\forall h:M\rightarrow\Real\text{ continuous },\frac{1}{n}\sum_{k=0}^{n-1}h\circ f^k(p)\xrightarrow[n\To\oo]{}\int h d\mu\}.
\]

\begin{definition}
	An $f$ - invariant measure $\mu$ is physical if it is ergodic and $Leb(B(\mu))>0$.	
\end{definition}

If $(f,\mu)$ is NUH, then a celebrated result of Pesin \cite{LyaPesin} implies that $\mu$ can be written (in the $\omega^{\ast}$ topology) as the converging series of a sequence $\{\mu_n\}_{n\geq 1}$ of ergodic probability measures of $f$. The set of supports $\{\mathrm{supp}(\mu_n)\}_{n\geq 0}$ is a countable partition of $M$, hence there exists at least one $\mu_n$ for which its support has positive Lebesgue measure: this $\mu_n$ is a physical measure for $f$.

Establishing the existence of physical measures is of utmost interest in smooth ergodic theory; most diffeomorphisms do not leave invariant any smooth volume, but physical measures at least detect the Lebesgue class (which comes from the differential/Riemannian structure of the manifold) in terms of its generic points, and thus they can be thought as a reasonable substitute for conservativity.  Regrettably the list of known abundant examples (i.e., in an open class or at least, in parameterized families) having physical measures is not very large. Some illustrative results are \cite{BowenRuelle} (physical measures for hyperbolic diffeomorphism), \cite{Jakob} (physical measures for the quadratic family),  \cite{Bonatti2000,SRBMostexp} (physical measures for a class of wPH maps with controlled dynamics along the dominated bundle) and \cite{dynHenon,SRBHenon,Abundance} (physical measures for H\'enon maps).

\esp

In the significant particular case when $S_r$ is a surface map we obtain the following.

\begin{coroA}\label{coroA}
	Assume that $\{f_r=A_r\times_{\varphi_r} S_r\}_{r}$ is an adapted family, with $A_r\in SL(2,\Z)$ hyperbolic and $\{S_r:\mathbb{T}^2\rightarrow\mathbb{T}^2\}_r$ is a collection of  $\mathcal{C}^2$ conservative maps satisfying conditions \textbf{S-1},\textbf{S-2}. Let $\mu$ be the Lebesgue measure in $\Toro^4$.
	\begin{enumerate}
		\item[a)]  There exists $r_0$ such that for every $r\geq r_0$ there exists $Q(r)>0$ satisfying: if $(f_r,\mu)$ is ergodic then for $\mu$ almost every $p\in M$ and every $v\in T_pM\setminus\{0\}:$
		\[
		\lim_{n\To+\oo}\left|\frac{\log\norm{d_pf^n_r(v)}}{n}\right| >Q(r).
		\]
		Hence the map $f_r$ is NUH, and in particular has a physical measure. The same is true for any $\mu$ ergodic $\wtf$ in a $\mathcal{C}^2$ neighborhood $\mathcal{U}_r$ of $f_r$. 
		\item[b)] In general, if $f_r$ is not ergodic but is $\mathcal{C}^3$, then it is approximated in the $\mathcal{C}^2$ category by stably (ergodic) NUH diffeomorphisms. 
	\end{enumerate}
\end{coroA}

\begin{proof}
	Assume first that $\mu$ is ergodic for $f_r$. By Corollary 3.5 of \cite{LedraExp}, for $\mu$ almost every $p\in M$ and every $v\in T_pM\setminus\{0\}$ it holds  $\chi_{f_r}(m,v)\in\{\chi^u_{f_r},\chi^s_{f_r},\chi^{c,1}_{f_r},\chi^{c,2}_{f_r}\}$, where $\chi^u_{f_r},\chi^s_{f_r}$ correspond to the exponents in $E^{u}_{f_r}, E^s_{f_r}$ and $\chi^{c,1}_{f_r}\leq\chi^{c,2}_{f_r}$ are the center exponents. Since $f_r$ preserves the volume $\mu$, by  Proposition 4.3 in \cite{LedraExp} it holds
	\begin{equation}\label{eq.sumacerototal}
	\chi^u_{f_r}+\chi^s_{f_r}+\chi^{c,1}_{f_r}+\chi^{c,2}_{f_r}=0.
	\end{equation}
	
	On the other hand, for every $x\in\Toro^2$ the Jacobian of the restriction $f_r|_{\{x\}\times \Tor}:\{x\}\times \Tor\rightarrow\{A_r(x)\}\times\Tor$ is equal to one, hence again by the same Proposition we have that
	\begin{equation}\label{eq.sumacerocen}
	\chi^{c,1}_{f_r}+\chi^{c,2}_{f_r}=0
	\end{equation}
	and thus
	\begin{equation}\label{eq.sumacero}
	\chi^u_{f_r}+\chi^s_{f_r}=0.
	\end{equation}
	
	Note that the family $\{f_r\}_r$ satisfies the hypotheses of our Main Theorem, thus for $r$ sufficiently large there exists $Q(r)>0$ so that $\chi^{c,2}_{f_r}\geq Q(r)$, which implies by \eqref{eq.sumacerocen} that $\chi^{c,1}_{f_r}\leq -Q(r)$, and $(f_r,\mu)$ is NUH.
	
	To consider perturbations of $f_r$, we observe that since $E^u_{f_r}, E^s_{f_r}$ are one-dimensional, and due to ergodicity of the system $(f_r,\mu)$, we can write 
	\begin{align*}
	\chi^u_{f_r}=\int \log\norm{d_pf_r|E^u_{f_r}} d\mu(p)\\
	\chi^s_{f_r}=\int \log\norm{d_pf_r|E^s_{f_r}} d\mu(p).
	\end{align*}
	The stable and unstable bundles in the partially hyperbolic class depends $\mathcal{C}^1$ continuously on the map (cf. Theorem 2.15 in \cite{HPS}), hence there exists $\mathcal{N}_r$ a $\mathcal{C}^1$ neighborhood of $f_r$ so that if $\wtf\in \mathcal{N}_r$ preserves $\mu$ and is ergodic, then
	\begin{itemize}
		\item the sum of all its exponents (with respect to $\mu$) is zero (by the same argument used for $f_r$ cf. \eqref{eq.sumacerototal}).
		\item $|\chi^u_{\wtf}+\chi^s_{\wtf}|<\frac{Q(r)}{2}$ (by equality \eqref{eq.sumacero} and continuity of the quantities).
	\end{itemize}	
	By the above, $\displaystyle{\chi^{c,2}_{\wtf}+\chi^{c,1}_{\wtf}<\frac{Q(r)}{2}}$ as well. If $\mathcal{U}_r$ is the $\mathcal{C}^2$ neighborhood of $f_r$ given in the Main Theorem, we hence deduce
	\[
	\wtf\in \mathcal{U}_r\cap V_r\Rightarrow \chi^{c,2}_{\wtf}>Q(r)\Rightarrow \chi^{c,1}<-\frac{Q(r)}{2}<0
	\]
	and $(\wtf,\mu)$ is NUH. This concludes the first part.
	
	For the second we use a Theorem of K. Burns and A. Wilkinson \cite{StErgSkew2} that allow us to approximate any $\mathcal{C}^3$ conservative skew product like $f_r$ (in the $\mathcal{C}^2$ topology) by conservative maps $\wtf$ having a $\mathcal{C}^2$ neighborhood where every $\mu$ preserving diffeomorphism on it is ergodic, which together with part a) implies the result.
\end{proof}

Part b) is an interesting consequence, since it does not require a priori knowledge of ergodicity of the maps. We remark that even though there exists a characterization of ergodicity for skew-products \cite{anzai1951}, the conditions in practice are very difficult to check.

More interestingly, the Main Theorem can be used to establish directly non uniform hyperbolicity of several examples, and then try to use this information (together with the fact that the exponents are uniformly separated from zero) to prove ergodicity. This has recently been achieved by D. Obata \cite{Obata2018} for the original \cite{NUHD} map, who proves in fact that the map is $\mathcal{C}^2$ - stably Bernoulli (a much stronger condition than ergodicity). His arguments are quite sophisticated.

\subsection{Examples: Random products of standard maps}

In this part we present several examples of applications of the Main Theorem. We also compare the results with some of the available literature.

\subsubsection{The Standard Map.} Let us consider again the map $f_r$ introduced in \cite{NUHD},
\begin{equation*}
f_r(x,y,z,w)=(A^{[2r]}\cdot(x,y),s_r(z,w)+P\circ A^{[r]}\cdot(x,y))\quad P(x,y)=(x,0).
\end{equation*}
To prove that (for large $r$) the map $f_r$ is NUH, the arguments of the aforementioned article make extensive use of the following properties.

\begin{itemize}
	\item One can obtain sharp estimates on the form of the (one dimensional) strong unstable bundle.
	\item The center bundle is two dimensional.
\end{itemize}

Here we generalize the results of \cite{NUHD}, allowing higher-dimensional center behavior.  We will start by showing that $s_r,s_r^{-1}$ satisfy conditions \textbf{S-1},\textbf{S-2}, and that the coupling in $\{f_r\}_r$ is adapted. We do so because we will be using similar types of couplings in our other examples, and because we are interested in dynamics related to the standard family.

The proof of conditions \textbf{S-1},\textbf{S-2} for $s_r$ essentially follow from the computations carried in Section 4 of \cite{NUHD}, and are recalled below. A direct computation shows that  
\begin{equation}\label{eq.deristandard}
d_{(x,y)}s_r=\begin{pmatrix}
\Omega_r(x,y)& -1\\
1& 0\\
\end{pmatrix}
\end{equation}
where $\Omega_r(x,y)=2+r\cos(x)$. For $z=(x,y)$ the vectors $\vec{v}_r(z)=(1,\Omega_r(z))$, $\vec{h}_r(z)=(\Omega_r(z),-1)$ form an orthogonal basis of $T_z\Tor$, and one checks that 
\[
d_zs_r(\vec{v}_r(z))=(0,1),\quad d_zs_r(\vec{h}_r(z))=(1+\Omega^2_r(z),\Omega_r(z)).
\]
With this it is easy to verify that if $X$\ is a unit vector field in $\mathbb{T}^2$,  
\begin{equation*}
\norm{d_zs_r(X_z)} \geq [r]\cdot|\sin\theta^X(z)|\cdot|\cos{ x}|-2
\end{equation*}
where $\theta^X(z)$ is the angle $\angle(X_z,\vec{v}_r(z))$. We define the critical strip $\mathit{Crit}(s_r)$ as the set of points
\begin{equation}\label{critica}
\mathit{Crit}(s_r)=\{z=(x,y)\in\mathbb{T}^2:b_1\leq x\leq b_2\text{ or }b_3\leq x\leq b_4\}	
\end{equation}
where $0<b_1<\pi/2<b_2<b_3<3\pi/2<b_4<2\pi$\ are such that
\begin{equation*}
\cos{ b_1}=\cos{ b_4}=\frac{1}{\sqrt{r}}\qquad \cos{ b_2}=\cos{ b_3}=-\frac{1}{\sqrt{r}}.
\end{equation*}
It holds that $l(\mathit{Crit}(s_r))\leq \frac{8}{\sqrt{r}}$ (cf.\@ (14) in \cite{NUHD}; note that $\mathit{Crit}(s_r)$ is the union of two connected bands). We consider also the cone 
\begin{equation}
\Delta_r:=\mathbb R\cdot \{(1,n)\in\Real^2:\; |n|\le \sqrt[4]  r\}.
\end{equation}

As for the sets $\mathcal{B}_r^{+}, \mathcal{B}_r^{-}$, they are treated in Lemma 6 of \cite{NUHD}: if $\mathrm{C}\Delta^+_r=\Real\cdot\{(1,n):n\geq \sqrt[4]{r}\}, \mathrm{C}\Delta^-_r=\Real\cdot\{(1,n):n\leq -\sqrt[4]{r}\}$ then $\mathcal{B}_r^{+}=\{(x,y)\in\Tor: \cos x<0\}, \mathcal{B}_r^{-}=\{(x,y)\in\Tor: \cos x>0\}$. It follows that $l(\mathcal{B}_r^{+}), l(\mathcal{B}_r^{-})=\pi$.

Using $\mathit{Crit}(s_r)$ as critical region, the bands $\mathcal{B}_r^{+},\mathcal{B}_r^{-}$  and the cone $\Delta_r$ one can verify the following.

\begin{lemma}\label{lem.standard}
	The maps $s_r,s_r^{-1}$ satisfy \textbf{S-1}, \textbf{S-2}. Moreover,
	\begin{align}
	\beta(r)&\geq r^{1/6}-2\\
	\zeta(r)&\geq 1/2r
	\end{align}
	In particular we can take $\sigma=2$.
\end{lemma}

\begin{proof}
	The bound for $\beta(r)$ is spelled in Lemma 4 of \cite{NUHD}, while the bound for $\zeta(r)$ is simple to obtain from \eqref{eq.deristandard}. As for \textbf{S-2}, invariance of $\Delta_r$ is proved in Lemma 5 of \cite{NUHD}, while restitution of the expansion direction is implicit in Lemma 6 of the same work: indeed if $\vec{v}\in \mathrm{C}\Delta_r\subset T_z\Tor$ then $\vec{v}=c\cdot(0,1)$ and hence $d_zs_r(\vec{v})=c\cdot(-1,0)\in \Delta_r$, or $\vec{v}=c\cdot (1,n)$ where either $n\geq \sqrt[4]{r}$ ($\vec{v}\in \mathrm{C}\Delta_r^+$), or $n\geq \sqrt[4]{r}$ ($\vec{v}\in \mathrm{C}\Delta_r^+$). Since
	\[
	d_zs_r(\vec{v})=c\cdot(2+r\cos x-n,1)
	\]
	we deduce that if $\cos x$ and $-n$ have the same sign (and $r$ large), this vector makes a small angle with the horizontal axis in $\Real^2$ and thus is contained in $\Delta_r$ away from its boundary. It follows that if $\vec{v}\in \mathrm{C}\Delta_r^+, z\in \mathcal{B}_r^+$ or $\vec{v}\in \mathrm{C}\Delta_r^-, z\in \mathcal{B}_r^-$
	then $d_zs_r(\vec{v})\in \Delta_r$. Finally, note that $\Delta_r$ increases with $r$, which together with the previous remark implies \textbf{S-2}.
	
	For the inverse map one can use that $s_r^{-1}=R\circ s_r\circ R$ where $R(x,y)=(y,x)$. This well known fact is recalled in Lemma 1 of \cite{NUHD}
\end{proof}

Now we discuss the correlation function.

\begin{lemma}\label{lem.stdadapted}
	The coupling in $\{f_r\}_{r}$ is adapted.
\end{lemma}

\begin{proof}
	Let $\lambda=\norm{A|E^u_A},\tau=\norm{A|E^s_A}$. Since $A_r=A^{[2r]}$ we obtain $E^{u}_{A_r}=E^{u}_A, E^{s}_{A_r}=E^{s}_A$ and $\lambda_r=\lambda^{[2r]},\tau_r=\tau^{[2r]}=\frac{1}{\lambda_r}$. Observe that $\varphi_r(x,y)=P\circ A^{[r]}\cdot(x,y)$.
	
	The bundle $E^u_A$ makes a positive angle with the line $\{(0,y):y\in\Real\}\subset\Real^2$, hence $\norm{P|E^u_A}>0$ which in turn implies $\norm{P\circ\varphi_r|A_{r}}>0$. Note also that
	\[
	0<\norm{P\circ \varphi_r|E^u_A}=\norm{\varphi_r|E^u_A}\leq \norm{A^{[r]}}\leq \lambda^{[r]}.
	\]
	Using this and the fact that $\norm{ds_r},\norm{ds_r^{-1}},\norm{d^2s_r}\leq 2r$, both conditions \textbf{A-1}, \textbf{A-2} follow.
\end{proof}

Recall that two diffeomorphisms $f_1,f_2$ of a manifold $X$ are said to be differentiably conjugate if there exists a diffeomorphism $h:X\rightarrow X$ (the conjugacy) such that $f_1=h^{-1}\circ f_2\circ h$. If $f_1,f_2$ are conservative and differentiably conjugate, then by the chain rule their Lyapunov exponents coincide Lebesgue almost everywhere. We will employ this remark to deal with the inverse, and in particular we use the fact that $s_r$ is differentiably conjugate to its own inverse, i.e. it is \emph{reversible} (cf. Lemma \ref{lem.standard} above).  

Indeed, this property of $s_r$ is used in Lemma 1 of \cite{NUHD} to show that $f^{-1}_r$ is differentiably conjugate to the map
\begin{equation}\label{eq.fhat}
\widehat{f}_r(x,y,z,w)= (A^{-[2r]}\cdot(x,y),s_r^{-1}(z,w)+P\circ A^{-[r]}\cdot(x,y)):
\end{equation}
to be precise $f^{-1}_r=\widehat{R}\circ \widehat{f}\circ \widehat{R}$ where $\widehat{R}:\Toro^4\to\Toro^4$ is the involution
\begin{equation}\label{eq.wideR}
\widehat{R}(x_1,x_2,x_3,x_4)=(x_1,x_2,x_4,x_3).
\end{equation}
Observe that $\widehat{R}$ leaves invariant the bundle $V$, and thus the Main Theorem can be applied to $f_r^{-1}$. Reversibility of the center dynamics will also be used in the other examples to  deal with the inverse map. We remark that this property for area preserving maps seems to be rather common. See \cite{Reversible}.

\begin{remark}
	One can check that $\{f_r^{-1}\}_{r}$ satisfies condition \textbf{A-3} of the Appendix, but regrettably it does not satisfy condition \textbf{A-4}, and hence we cannot use Corollary B of the Appendix to conclude directly that for $r$ large the map $f_r$ is NUH. 	
\end{remark}

We have proved the following.

\begin{corollary}[Main Result of \cite{NUHD}]\label{NUHD}
	For large $r$ the map $f_r$ is $\mathcal{C}^2$ - robustly NUH: for such $r$ there exists $\mathcal{U}_r$ $\mathcal{C}^2$ neighborhood of $f_r$ with the property that if $g\in \mathcal{U}_r$ is conservative then $g$ is NUH. The map $g$ has a physical measure and moreover for Lebesgue almost every $p$, it holds
	\[
	v\in T_pM\setminus\{0\}\quad\Rightarrow\lim_{n\To\oo}\left|\frac{\log\norm{d_pf^n(v)}}{n}\right|>\frac{3}{5}\log r. 
	\]	
\end{corollary}
The lower bound in the exponents is better than the one in \cite{NUHD} but worse to the one obtained by Obata \cite{Obata2018}. See Remark \ref{bound}.

\esp

\begin{remark}
	Note that for maps $\wtf$ in $\mathcal{U}_r$ above, the splitting between the center Oseledet's subspaces is not dominated ($d\wtf|E^c_{\wtf}\simeq ds_r|V$, and the later cannot be dominated due to the existence of $\mathit{Crit}$). Compare with the result of J. Bochi and M.Viana \cite{Bochi2005} that states: for compact manifolds of dimension greater than one, $\mathcal{C}^1$-generically in the space of conservative diffeomorphisms either there exists a zero exponent with multiplicity two, or the Oseledet's splitting is dominated, almost everywhere. 	
\end{remark}

Our methods seem to be particularly adequate to deal with the case where $S_r$ is given by the generalized family of standard maps studied by O. Knill in \cite{entKnill} 
\[(x,y)\rightarrow (2x-y+rV(x),x)\]
where $V(x)$ is a periodic Morse potential. It is interesting to compare the lower bound on the topological entropy $h_{top}(s_r)\geq \frac{1}{3}\log (C\cdot r)$ given by Knill and our lower bound on the positive center exponent of $f_r$. By the Pesin formula \cite{LyaPesin}, for conservative surface maps the metric entropy with respect to the area is equal to the integral of the largest Lyapunov exponent, and since the two obtained bounds are comparable this can be taken as evidence for the suitability of Lebesgue measure to detect the majority of chaotic behavior in the dynamics of $s_r$.

\subsubsection{Higher dimensional examples: random products of (coupled) standard maps.} 
We will now apply our techniques to coupled products of standard maps, thus giving examples of more general center (higher dimensional) dynamics. Let us go back to the map $g_{r}:\Toro^{2(e+1)}\rightarrow\Toro^{2(e+1)}$ given in \eqref{standardproduct}, 
\[
g_r(z_0,z_1,\cdots, z_e)=\Big(A^{[2r]}(z_0),s_r(z_1)+P\circ A^{[r]}(z_0),\cdots,s_r(z_e)+P\circ A^{[r]}(z_0)\Big).
\]
\begin{proof}[Proof of Theorem  \ref{teo.introcoupled} and Corollary \ref{cor.Frosty}]
	We start recalling that $dg_r|V=ds_r\times \cdots\times ds_r$ (the product taken $e$ times); from this fact and by the computations of Lemma \ref{lem.standard} we get that conditions \textbf{S-1}, \textbf{S-2} are satisfied for $g_r$. The corresponding correlation map $\widetilde{\varphi}_r:\Tor\to\Toro^{2e}$ is of the form 
	\[
	\widetilde{\varphi}_r(z)=\varphi_r(z)\times\cdots\times\varphi_r(z),
	\]
	where $\varphi_r$ is the correlation function associated to $f_r$ (considered in Lemma \ref{lem.stdadapted}). From this it follows directly that the coupling in $\{g_r\}_r$ is adapted. We can thus apply the Main Theorem to $\{g_r\}_r$ and ensure the existence of $Q(r)>0$ (if $r$ is sufficiently large), so that for Lebesgue almost every $p$, the sum of the multiplicity of center exponents at $p$ bigger than $Q(r)$ is at least $e$, and moreover this property is $\mathcal{C}^2$ robust among conservative diffeomorphisms close to $g_r$.  
	
	For the negative exponents, we consider $\Gamma:\Toro^{2(e+1)}\to \Toro^{2(e+1)}$ the map 
	\[
	\Gamma(z_0,z_1,\cdots, z_e)=(z_0,R(z_1),\cdots, R(z_e))
	\]
	where $R:\Tor\to\Tor$ is the involution $R(x,y)=(y,x)$. Then $\Gamma$ is a differentiable involution, it preserves the bundle $V$, and one can verify that $g^{-1}_r=\Gamma\circ \widetilde{g}_r\circ \Gamma$ where
	\[
	\widetilde{g}_r(z_0,z_1,\cdots, z_e)=\Big(A^{[-2r]}(z_0),s_r(z_1)+P\circ A^{-[r]}(z_0),\cdots,s_r(z_e)+P\circ A^{-[r]}(z_0)\Big)
	\]
	(compare \eqref{eq.fhat},\eqref{eq.wideR}). Note that $\widetilde{g}_r$ has the same form of $g_r$ thus by the Main Theorem applied to $\widehat{g}_r$, and by using  
	$\Gamma$ to conjugate every map sufficiently $\mathcal{C}^2$ close to $\widetilde{g}_r$ to a map $\mathcal{C}^2$ close to $g_r^{-1}$, we conclude the existence of a $\mathcal{C}^2$ neighborhood of $g_r^{-1}$ so that for every conservative $\widetilde{g}$ in this neighborhood,  for Lebesgue almost every $p$ the sum of the multiplicity of center exponents of $\widetilde{g}$ at $p$ that are bigger than $Q(r)$ is at least $e$.  Since the set of $\mathcal{C}^2$ diffeomorphisms of $M$ is a topological group when equipped with the $\mathcal{C}^2$ topology, we finally deduce the existence of a $\mathcal{C}^2$ neighborhood $\mathcal{U}_r$ of $g_r$ so that every conservative $g\in\mathcal{U}_r$ is NUH with respect to $\mu$. This establishes Theorem \ref{teo.introcoupled} of the Introduction.
	Now we look at the map $g_{r,\tau}:=A^{[2r]}\times_{\widetilde{\varphi}_r}u_{r,\tau}$ considered in Corollary \ref{cor.Frosty}, where $u_{r,\tau}:\Toro^4\to\Toro^4$ is given by
	\[
	u_{r,\tau}(x,y,z,w)=(2x-y+r\sin(x)+\tau\sin(x+z),x,2z-w+r\sin(z)+\tau\sin(x+z),z)
	\]
	and $\widetilde{\varphi}_r:\Tor\to\Toro^4$ is defined as above. Note that $g_{r,\tau}$ is a perturbation of $g_r$ (with $e=2$) if $\tau$ is small. In particular, for large $r$ and small $\tau$, $g_{r,\tau}$ is PH and  $\mathcal{C}^2$ robustly NUH, as claimed in the Corollary.
\end{proof}

\esp

The maps $g_r, g_{r,\tau}$ treated above provide new interesting higher dimensional examples to add to the list of known NUH systems, enabling us to construct partially hyperbolic diffeomorphisms with rich center dynamical behaviors. Other interesting examples can be obtained by considering some symplectic twist maps in $\Toro^{2d}$, and in particular those of the form 
\[
S_V(q,p)=\big(2q-p+\nabla V(q),q\big)\quad q,p\in \Toro^d
\]
where $V\in \mathcal{C}^2(\Toro^d,\Real)$. We illustrate this with the potential $V_{\tau_1,\tau_2,\tau_3}:\Toro^2\rightarrow\Toro^2$ where
\begin{equation}
V_{\tau_1,\tau_2,\tau_3}(x,y)=\tau_1\cos(x)+\tau_2\cos(y)+\tau_3\cos(x+y).	
\end{equation}
The resulting map $S_{V_{\tau_1,\tau_2,\tau_3}}$ is the so called three parameter Froeschl\'e family \cite{Froeschle}, and it is very similar to $u_{\tau,r}$. Let us write 
$\varphi_r(x_1,x_2):=(P\circ A^{[r]}(x_1,x_2),P\circ A^{[r]}(x_1,x_2))$: a slight variation in the argument used in the proof above (choosing $\tau_1,\tau_2\approx r$ large, $\tau_3$ small) yields the following.

\begin{theorem}
	For every $r$ sufficiently large there exists an open set in the parameter space $E_r\subset\{(\tau_1,\tau_2,\tau_3):\tau_i\in\Real\}$ and $C>0$ such that if $(\tau_1,\tau_2,\tau_3)\in E_r$ then $F_{r,\tau_1,\tau_2,\tau_3}:=A^{[2r]}\times_{\varphi_r}S_{V_{\tau_1,\tau_2,\tau_3}}$ is PH and satisfies for Lebesgue almost every $p$ and $v\in T_pM\setminus\{0\}$, 
	\[
	\lim_{n\To\oo}\Big|\frac{\log\norm{d_pF_{r,\tau_1,\tau_2,\tau_3}^n(v)}}{n}\Big|>C\log(r).
	\]
	The same holds for every conservative map in a $\mathcal{C}^2$ neighborhood of $F_{r,\tau_1,\tau_2,\tau_3}$.
\end{theorem}
For further information on this family and other symplectic twist maps we refer the reader to \cite{Gole2001}.   

To the extent that the author could check none of these examples can be treated by any available technique in the literature, and is his hope that he provided enough evidence to convince the reader of the versatility in the method presented. It is also worth noticing that for establishing our results it is not necessary to assume a priori ergodicity of the map. In particular, we do not need to perturb to guarantee accessibility as the methods based on the invariance principle \cite{ExtLyaExp} usually require. This is important by two reasons: first because our goal is understand concrete examples more than their perturbations (i.e.\@ random products of the standard map, versus  perturbations of random products of the standard map), and second because as we mentioned in the introduction, establishing NUH of the system could be used in some cases to establish ergodicity \cite{Obata2018}. 

\subsection{Cocycles over the shift}

Systems exhibiting similar dynamics with PH skew products are cocycles over expanding endomorphisms, and in particular over shifts spaces. These kind of maps are also very popular in the literature. As an example, we consider for a positive integer $k$ the complete shift space in $k$ symbols
\[
\Sigma_k^{+}=\{0,\ldots,k-1\}^{\mathbb{N}}
\]
Equipped with the natural product topology, it is a compact metrizable space. The dynamics is given by the shift map $\sigma(\underline{\omega})_n=\omega_{n+1}$, and we consider the Bernoulli measure $\mu_k=(1/k,\ldots, 1/k)^{\mathbb{N}}$ on $\Sigma_k^{+}$, i.e.\@ the product measure obtained by assigning weight $1/k$ to each symbol.  Given a diffeomorphism $S:\mathbb{T}^d\rightarrow \mathbb{T}^d$ and a continuous function $\varphi: \Sigma_k\rightarrow \mathbb{T}^d$ we can define a (continuous) cocycle of matrices as follows: let $G_{\varphi}:\Sigma_k\times \mathbb{T}^d\rightarrow \Sigma_k\times \mathbb{T}^d$ be the map
\[
G_{\varphi}(\underline{\omega},x)=(\sigma\underline{\omega},\varphi(\underline{\omega})+S(x))
\]
and for $n\geq 0$ denote $g^n(\underline{\omega},x)$ the projection on the second coordinate of $G_{\varphi}^n(\underline{\omega},x)$. Then define
\begin{equation}
\partial_cG_{\varphi}^n(\underline{\omega},x):=d_{g^{n-1}(\underline{\omega},x)}S\circ\cdots\circ d_{x}S.
\end{equation}
This way we have determined a cocycle of matrices which could be interpreted as the derivative cocycle of $G_{\varphi}$.

In a recent work by A. Blumenthal, J. Xue and L.S Young \cite{LyaRandom}, the authors considered a similar type of random cocycle over the infinite shift $T:(-\ep,\ep)^{\mathbb{N}}\rightarrow (-\ep,\ep)^{\mathbb{N}}$ ($\ep>0$ small), with fiber maps $\psi_r(\omega,(x,y))=s_r(x+\omega_0 \mod ,y)$, and other $2$-dimensional conservative maps satisfying certain expanding conditions in the spirit of ours \textbf{S-1}, \textbf{S-2}. For these systems the authors prove the positivity (with precise bounds) of the largest exponent of the cocycle for $\nu^{\mathbb{N}}$ almost everywhere, where $\nu$ is the uniform measure on $(-\ep,\ep)$. Their techniques are more probabilistic in nature than ours, and seem to depend on the two-dimensionality of the fiber maps (hence, one can deal with one Lyapunov exponent only).

It is a simple exercise in dynamical systems to show that $\sigma:\Sigma_k^{+}\rightarrow\Sigma_k^{+}$ equipped with $\mu_k$ is (measure theoretically) conjugate to the expanding map $E_k:\Toro\rightarrow\Toro$, $E_k(\theta)=k\cdot\theta\mod 2\pi$ with the Lebesgue measure on $\Toro$. Thus, instead of cocycles over $\Sigma_k^{+}$ we can equivalently consider cocycles over $E_k$.

For $k\in\mathbb{Z}\setminus\{1,0,-1\}$ consider the multiplication map $E_k$. Smale's solenoid construction (described for example in Chapter 7 of \cite{Robinson1998}) permits us to find a diffeomorphism  $L_{k}:N:=\Toro\times D\rightarrow N, D=\{(x,y)\in\Real^2:x^2+y^2\leq 1\}$ such that 
\begin{itemize}
	\item $L_k$ has an hyperbolic (expanding) attractor $\Lambda\subset N$. 
	\item $L_k(\theta,(x,y))=(E_k(\theta),\psi(x,y))$ for some $\psi:N\rightarrow D$.
\end{itemize} 

Let us remind the reader that a  diffeomorphism $L:N\rightarrow N$ is said to have an hyperbolic attractor $\Lambda$ if 
\begin{itemize}
	\item $\Lambda$ is an hyperbolic set for $L$.
	\item There exists $U\subset N$ open such that $\Lambda=\cap_{n\in\mathbb{N}}L^n(U)$.
\end{itemize}
If $\Lambda$  is an hyperbolic attractor and $\Fu$ denotes its unstable lamination, one verifies easily that  $\Wu{x}\subset \Lambda$ for every $x\in \Lambda$. See \cite{Shub} Chapters 5 and 6. 

Furthermore, in case that $L$ is of class $\mathcal{C}^2$, there exists a particularly important invariant measure $\mu_{SRB}$ supported on $\Lambda$ which can be characterized as follows: $\mu_{SRB}$ is the unique $L$ - invariant measure such that for any sufficiently small lamination chart $B$ corresponding to $\mathcal{W}^u_{\Lambda}$, its conditional measures on $\Wu{x}\cap B$ are absolutely continuous with respect to the induced Lebesgue measure, for almost every $x\in B$. $\mu_{SRB}$ is the Sinai-Ruelle-Bowen measure of $\Lambda$. See \cite{SRBattractor,EquSta} for further information on this topic. 

\begin{remark}\label{main.atractor} All the arguments used for the proof of the Main Theorem adapt directly to families of skew products $\{f_r=A_r\times_{\varphi_r} S_r\}_r$ where 
	\begin{itemize}
		\item $A_r=L_r|\Lambda_r$ where $L_r:\Toro^b\rightarrow \Toro^b$ is a diffeomorphism of class $\mathcal{C}^2$, and $\Lambda_r$ is an hyperbolic attractor for $L_r$ with one dimensional unstable bundle $E^u_{A_r}$.
		\item There exists a continuous Riemannian metric on $\Lambda_r$ such that $dA_r|E^u_{A_r}$ is conformal with respect to the associated norm, i.e. 
		there exists $\lambda>1$ such that $\norm{dA_r(v)}=\lambda\norm{v}$ for every $v\in E^u_{A_r}$.	
		\item $\{S_r=S_{r,1}\times\cdots\times S_{r,e}:\Toro^{2e}\to \Toro^{2e}\}_r$  satisfies \textbf{S-1}, \textbf{S-2}.
		\item The correlation functions $\varphi_r:\Lambda_r\to\Toro^{2e}$ are restriction of linear maps and properties \textbf{A-1}, \textbf{A-2} are satisfied. 
	\end{itemize}		
	In such case, our techniques give that  for $r$ sufficiently large there exists a $\mu_{SRB}\times Leb$ full measure set $NUH_r\subset \Lambda_r\times\Toro^{2e}$ and $Q(r)>0$ such that for $p\in NUH_r$ the sum of the multiplicity of center exponents at of $f_r$ at $p$ that are bigger than $Q(r)$ is at least $e$. Here we make a small abuse of language and call \emph{center exponents} to those associated to vectors tangent to $V=\{0\}\times T\Toro^{2e}$ (although the map $f_r$ is not PH in general).
	
	The same property remains valid replacing $S_r$ by a conservative diffeomorphism $S_r':\Toro^{2e}\to\Toro^{2e}$, provided that $S_r'$ is sufficiently $\mathcal{C}^2$ close to $S_r$.
\end{remark}	

Back to the solenoid, if $E^u$ denotes the unstable bundle of $A:=L_k|\Lambda$, then for $p=(\theta,(x,y))\in\Lambda$, the line $E^u(p)$ is a graph over $T_{\theta}\Toro\times\{0\}$, hence there exists a Riemannian metric $\norm{\cdot}$ on $\Lambda$ so that $\norm{dA(v)}=k\norm{v}$ for every $v\in E^u$ (in other words, $dA|E^u$ is conformal). See Proposition 7.5 in \cite{Robinson1998}. Define the map $t_r:\Toro\times \Tor\rightarrow \Toro\times \Tor$ with 
\begin{align}
t_r(\theta,x,y)=\Big(E_{k^{[2r]}}(\theta),s_r(x,y)+(E_{k^{[r]}}(\theta),0)\Big).	
\end{align}
and write $\partial_ct_r(\theta,x,y)$ for the derivative of $t_r|:\{\theta\}\times\Tor\rightarrow\{E_{k^{[2r]}}(\theta)\}\times\Tor$ at the point $(x,y)$.

\begin{proposition}\label{teoshift}
	There exists $r_0$ such that for every $r\geq r_0$ it holds for Lebesgue almost every $\theta\in\Toro,(x,y)\in\Toro^2$ 
	\begin{align*}
	\lim_{n\To+\oo}\frac{\log \norm{\partial_ct_r^n(\theta,x,y)}}{n}>\frac{3}{5}\log r
	\end{align*}	
\end{proposition}

\begin{proof}
	Consider $\widetilde{t}_r:N\times \Toro^2\rightarrow N\times \Toro^2$,
	\[
	\widetilde{t}_r(\theta,z,w,x,y)=\big(L_{k^{[2r]}}(\theta,z,w),s_r(x,y)+(E_{k^{[r]}}(\theta),0)\big).
	\]
	Arguing analogously as in Lemma \ref{lem.stdadapted} one checks that the coupling functions $\{\phi_r: \Lambda_r\to \Toro^2\}_r$ with $\phi_r(\theta,z,w,x,y)=(E_{k^{[r]}}(\theta),0)$ satisfy conditions \textbf{A-1}, \textbf{A-2}, and thus by the Remark \ref{main.atractor} we conclude that for $\displaystyle{\mu_{SRB}\times Leb-a.e.}\left((\theta,z,w,x,y)\in N\times \Toro^2\right)$ it holds
	\begin{align*}
	\lim_{n\To+\oo}\frac{\log \norm{\partial_c\widetilde{t}_r^n(\theta,x,y)}}{n}>\frac{3}{5}\log r
	\end{align*}
	where $\partial_c\widetilde{t}_r^n(\theta,x,y)$ is the derivative of
	\[\widetilde{t}_r^n|:\{(\theta,z,w)\}\times\Tor\rightarrow\{A_r^n(\theta,z,w)\}\times\Tor\] at the point $(x,y)$. As $\partial_c\widetilde{t}_r^n(\theta,x,y)=
	\partial_ct_r^n(\theta,x,y)(=ds_r(x,y))$, we conclude 
	\begin{align*}
	\lim_{n\To+\oo}\frac{\log \norm{\partial_ct_r^n(\theta,x,y)}}{n}>\frac{3}{5}\log r\quad \mu_{SRB}\times Leb-a.e.(\theta,z,w,x,y)\in N\times \Toro^2
	\end{align*}
	If $\pi:N\times \Toro^2\rightarrow \mathbb{T}\times \Toro^2$ is the projection map $\pi(\theta,z,w,x,y)=(\theta,x,y)$, then $\pi$ semi-conjugates $\widetilde{t}_r$ with $t_r$ while $\pi_{\ast}(\mu_{SRB}\times Leb)$ is Lebesgue on $\Toro^3$. From here it follows. 
\end{proof}

Likewise, one can obtain similar results considering the product of $e$ coupled standard maps instead of just one, establishing that the sum of the multiplicity of positive center exponents is at least $e$.  Since $E_k$ is conjugate to the one-sided shift in $k$ symbols, Theorem \ref{teoshift} provides a geometrical proof of the existence of non-zero Lyapunov exponents for cocycles over the shift. 

The type of system described above is similar to the ones considered by M. Viana in \cite{Multinonhyp}, although those were non-conservative. It is possible that some of the techniques presented here can be adapted to non-conservative case as well, but the author has not pursued that enterprise in this paper.

\esp

Besides the interest per-se in the dynamics of cocycles of standard maps, these systems appear naturally in physical and mathematical applications. Historically, they were introduced to study Arnold diffusion for systems of coupled oscillators \cite{chirikov1979}, albeit their research was mainly numerical. For a more up to date and rigorous study on this topic the reader is referred to the work of O. Castej\'on and V. Kaloshin \cite{randomitera}, where the authors analyze statistical properties of random products of standard maps appearing (as certain induced dynamics) in Arnold's original diffusion example. 

\section{Admissible curves and adapted fields}

The most important technical tools introduced in \cite{NUHD} are admissible curves and adapted fields. Roughly speaking, an adapted curve is a segment of the strong unstable manifold that makes approximately a complete turn in around every vertical torus $\{0\}\times\Tor_i$ while doing many more turns around the base torus $\mathbb{T}^2\times\{0\}$; an adapted field is a vector field along an admissible curve with very small variation. In that work however, the notion of admissible curve is tailored to the specific example considered. Here we present a more abstract definition.

\subsection{Partial Hyperbolicity}\label{parcialmente}

Let $\{f_r=A_r\times_{\varphi_r} S_r\}_{r}$ be a coupled family where $\{S_r\}_r$ satisfies \textbf{S-1}, \textbf{S-2}. We will be able to work from the beginning with perturbations of $f_r$. The map $f=f_r:M=\Toro^{2}\times\Toro^{2e}\rightarrow\Toro^{2}\times\Toro^{2e}$ is of the form
\[
p=(x,y)\in M\Rightarrow f(x,y)=(A(x),S(y)+\varphi(x))
\]
and we consider $\wtf=f+h$ a small $\mathcal{C}^2$ perturbation of it. Note that the derivative of $f$ can be expressed in block form as
\begin{equation}\label{derivadaf}
d_pf=\begin{pmatrix}
A & 0  \\
\varphi & d_yS 
\end{pmatrix}.
\end{equation}

\esp

Let us recall we are assuming that $A=A_r\in SL(2,\Z)$ is hyperbolic with associated decomposition $\Real^2=E^u_A\oplus E^s_A$, where $E^u_A=\Real\cdot e^u_A, E^s_A=\Real\cdot e^s_A$, $\norm{e^u_A}=\norm{e^s_A}=1$ and
\begin{align*}
&\lambda=\norm{A|E^u_A}=\norm{A(e^u_A)}\\
&\tau=\norm{A|E^s_A}=\norm{A(e^s_A)}=\frac{1}{\lambda}.
\end{align*}

The bundles $E^u_A,E^s_A$ determine continuous bundles on $M$, which by an innocuous abuse of language will be denoted by the same letters. Observe that $TM=E^u_A\oplus(E^s_A\oplus V)$, with $V=\{0\}\times T\Toro^{2e}$. For convenience we will use in $V=\bigoplus_{i=1}^e\Real^2_i$ the $\ell^{\oo}$ norm associated to this decomposition. Define
\begin{equation}
E=E_r=E^s_A\oplus V
\end{equation}
so that $TM=E^u_A\oplus E$, and note $\norm{df|E}\leq\norm{dS}+\norm{\varphi|E^{s}_A}+\norm{A|E^{s}_A}$. Define also 
\begin{equation}
\xi(r):=\frac{2\norm{\varphi|E^u_A}}{\lambda-\norm{df|E}}.
\end{equation}
By condition \textbf{A-1} in the coupling, the above quantity is positive and converges to zero as $r\To\oo$. Finally, consider the cone field of size $\xi(r)$ centered around $E^u_A$ in $E^u_A\oplus E$
\[
\Delta^{u}(p)=\{(v,w)\in E^{u}_A(p)\oplus E(p):\norm{w}<\xi\norm{v} \}\cup\{(0,0)\}.
\]
\begin{lemma}\label{accioninestable}
	There exists $r_1>0$ such that for every $r\geq r_1$, for every $p\in M$ it holds 
	\begin{itemize}
		\item $d_pf\left(\overline{\Delta^{u}}(p)\right)\subset \Delta^{u}(f(p))$, and
		\item $\forall X\in \Delta^{u}(p),\ \lambda\left(\frac{1-\xi(r)}{1+\xi(r)}\right)\norm{X}\leq \norm{d_pf(X)}\leq \lambda(1+\xi(r))\norm{X}$. 
	\end{itemize}
\end{lemma}

\begin{proof}
	Let $(v^u,w)\in \Delta^u(p)$ with $0\neq v^u\in E^u_A,w\in E$. We obtain
	\begin{align*}
	\norm{d_pf(w)+\varphi(v^u)}&\leq\norm{d_pf|E}\cdot\norm{w}+\norm{\varphi|E^u_A}\norm{v^u}
	\leq (\xi(r)\norm{d_pf|E}+\norm{\varphi|E^u_A})\cdot \norm{v^u}\\
	&<\lambda\xi(r)\norm{v^u}= \xi(r)\norm{A(v^u)}
	\end{align*}
	which shows the first part. As for the second, take $X=(v^u,w)\in \Delta^u(p)$ and use the previous inequality to compute
	\begin{align*} 
	&\norm{d_pf(X)}\leq\norm{Av^u}+\norm{d_pf(w)+d_p\varphi(v^u)}\leq\lambda(1+\xi(r))\norm{v^u}\leq \lambda(1+\xi(r))\norm{X}\\
	&\norm{d_pf(X)}\geq\norm{Av^u}-\norm{d_pf(w)+d_p\varphi(v^u)}\geq \lambda(1-\xi(r))\norm{v^u}\geq \lambda\frac{1-\xi(r)}{1+\xi(r)}\norm{X}
	\end{align*}
	where in the last part we have used $\norm{X}\leq \norm{v^u}+\norm{w}\leq(1+\xi(r))\norm{v^u}$.
\end{proof}

It is a well known consequence of the above that there exists a $df$ - invariant bundle $E^u_f\subset \Delta^u$ (see for example \cite{CroPotPH}), and in particular this implies that $f$ is wPH. Similarly, computing the derivative of $f^{-1}$ we obtain

\begin{equation}\label{derivadafinv}
d_pf^{-1}=\begin{pmatrix}
A^{-1} & 0  \\
-(d_{y'}S)^{-1} \circ \varphi\circ A^{-1} & (d_{y'}S)^{-1} 
\end{pmatrix},\quad f(x',y')=(x,y)=p
\end{equation} 

and thus if we define $E'_r=E'=E^u_A\oplus V$, $\displaystyle{\widehat{\xi}(r)}=\frac{\norm{dS^{-1}\circ\varphi|E^s_A}}{1-\tau\norm{df^{-1}|E'}}$ and 
\[
\Delta^{s}(m)=\{(v,w)\in E^{s}_A(m)\oplus E'(m):\norm{w}<\widehat{\xi}\norm{v} \}\cup\{(0,0)\}
\]
we can proceed as before and deduce (by \textbf{A-1}) that $f^{-1}$ has an invariant expanding bundle $E^s_f$ with 
\[
\tau\left(\frac{1-\widehat{\xi}(r)}{1+\widehat{\xi}(r)}\right)\leq m(df|E^s_{f})\leq \norm{df|E^s_{f}}\leq \tau(1+\widehat{\xi}(r)),
\]  
hence $f$ is PH. For future reference, we spell out this fact in form of a Corollary.

\begin{corollary}\label{partialhyp}
	There exists $r_1>0$ such that for $r\geq r_1$ it holds that $f=f_r$ is PH with invariant splitting $TM=E^u_f\oplus V\oplus E^s_f$.
\end{corollary}

\begin{proof}
	Once the existence of $E^u_f,E^s_f$ has been established, its expanding/contracting character is direct consequence of the previous Lemma and the inequality above. Since $E\cap E'=V$, and by hypotheses on the coupling the Whitney sums $E\oplus E^u_f, E'\oplus E^s_f$ are dominated for $f,f^{-1}$, the rest follows.
	
\end{proof}

From now on we will assume $r\geq r_1$. Note that $E^u_f\subset \Delta^{u}, E^s_f\subset \Delta^s$ and $\xi(r),\widehat{\xi}(r)\xrightarrow[r\To\infty]{}0$, thus:

\begin{lemma}\label{converfibf}
	the angles $\angle(E^u_f,E^u_A), \angle(E^s_f,E^s_A)$ converge uniformly to zero as $r\To+\infty$.
\end{lemma}

\esp

Partial Hyperbolicity is a $\mathcal{C}^1$ open condition (cf.\@ Theorem 3.6 in \cite{PesinLect}), thus for every $r$ there exists $\phi(r)>0$ such that if $\norm{h}_{\mathcal{C}^1}<\phi(r)$ then $\widetilde{f}=f+h$ is PH. Its $d\widetilde{f}$ - splitting $TM=E_{\wtf}^s\oplus E^c_{\wtf}\oplus E^u_{\wtf}$ converges to $E^u_f\oplus V\oplus E^s_f$ as $\norm{h}_{\mathcal{C}^1}\rightarrow 0$ in the corresponding Grassmanian. We define 

\begin{equation}
E_{\wtf}=E^u_A\oplus E^c_{\wtf}.
\end{equation}

\begin{remark}
	As $E^c_f=V$ is differentiable and integrates to a foliation by torii,  Theorem 7.1 of \cite{HPS} guarantees that the center bundle of any sufficiently small  $\mathcal{C}^1$ perturbation of $f$ also integrates to a (non-necessarily smooth) foliation by torii. 
\end{remark}

We deduce.

\begin{corollary}\label{angulowtf}
	The function $\phi(r)$ can be chosen so that for $r\geq r_1$, it holds 
	\begin{align*}
	&\angle(E^u_{\wtf},E^u_A))\leq \frac{\pi}{100}\\
	&\angle(E^s_{\wtf},E^s_A)\leq \frac{\pi}{100}.
	\end{align*}
\end{corollary}

It is also known (cf. Theorem 3.8 in \cite{PesinLect}) that the invariant bundles of $\wtf$ are H\"older continuous, provided that $\wtf$ is $\mathcal{C}^2$. We will
choose $\phi(r)$ sufficiently small so that $E_{\wtf}^c$ is $\theta$ - H\"older, with $\theta\approx 1$.

\esp

\subsection{Curves tangent to $E^u$}

Our goal is have a qualitative description of the bundle $E_{\wtf}^u$, and for this we use the graph transform method. By the (un)-stable manifold theorem $E^u_{\wtf}$ is integrable to an $\wtf$-invariant foliation $\Futf=\{\Wutf{p}\}_{p\in M}$ whose plaques can be obtained as graphs of functions from $E^u_{\wtf}$ to $E_{\wtf}$. The previous Corollary allow us to deduce that there exist $\delta>0$ and a continuous family $\{\widetilde{\Gamma}_p:E_A^{u}(p;\delta)\rightarrow E(p)\}_{p\in M}$ such that for every $p$ its local strong unstable manifold $\Wutf{p;loc}$ is of the form
\[
\Wutf{p;loc}=p+\mathrm{graph}(\widetilde{\Gamma}_p)
\]  
where $\norm{\widetilde{\Gamma}_p}_{\mathcal{C}^1}<1$. Here $E_{A}^{u}(p;\delta)$ denotes the $\delta$ - disc centered at zero inside $E_{A}^{u}(p)$. Similar notation will be used for other bundles. 

Recall that $K(r)$ is defined in condition \textbf{A-2}.

\begin{lemma}\label{accionderivada}
	There exists $\varpi(r)>0, \varpi(r)\xrightarrow[r\To+\oo]{}0$ with the following property. Consider $p\in M$ and $X=e^u_A+X^E\in E^u_{\wtf}(p)$ with $X^E\in E$. Then for $1\leq i,j\leq e,$  
	\begin{flalign*}
	1-\varpi(r)\leq&\frac{\norm{P_iX^E}}{\norm{P_jX^E}}\leq 1+\varpi(r).\\
	0<(1-\varpi(r))\frac{\norm{P_i\circ\varphi_r|E^u_A}}{\lambda}\leq&\frac{\norm{P_iX^E}}{\norm{X}}\leq (1+\varpi(r))\frac{\norm{P_i\circ\varphi_r|E^u_A}}{\lambda}.
	\end{flalign*}
\end{lemma}

\begin{proof}
	Fix $X=e^u_A+X^E\in E^u_{\wtf}(p)$ and note that $X^E=d_0\widetilde{\Gamma}_p(e^u_A)\in E$. By invariance of the unstable bundle, $X=d_{\wtf^{-1}p} \wtf(Y)$ where $Y\in E^u_{\wtf}(\wtf^{-1}p)$, hence of the form $Y=a\cdot e^u_A+Y^E$ where $a\in\Real$, $Y^E=a\cdot d_0\widetilde{\Gamma}_{\wtf^{-1}p}(e^u_A)\in E$. Since $\displaystyle{\norm{d_0\widetilde{\Gamma}_{\wtf^{-1}m}}\leq 1}$ it follows $\norm{Y^E}\leq|a|, \norm{Y}\leq 2|a|$. Recall that $E=E^s_A\oplus V$, thus $Y^E$ can be written as $Y^E=Y^s+Y^v$ where $Y^s\in E^u_A, Y^v\in V$. Let $\pi^u:TM\rightarrow E^u_A, \pi^E:TM\rightarrow E$ be the projections onto $E^u_A, E$ respectively.
	
	As $\displaystyle{e^u_A=\lambda\cdot a e^u_A+\pi^ud_{\wtf^{-1}p}h(Y)}$,
	\[
	1=\norm{\lambda\cdot a e^u_A+\pi^ud_{\wtf^{-1}p}h(Y)}
	\]
	and by using that $\norm{\pi^ud_{\wtf^{-1}p}h}\leq \phi(r)$ we get the following bound for $|a|$:
	\[
	\frac{1}{\lambda+2\phi(r)}\leq|a|\leq \frac{1}{\lambda-2\phi(r)}.
	\]
	
	On the other hand,  
	\begin{equation}\label{eq.elv}
	X^E=a\cdot \varphi_r(e^u_A)+\varphi_r(Y^s)+d_{\wtf^{-1}p}S_r(Y^v)+\pi^Ed_{\wtf^{-1}p}h(Y)
	\end{equation}
	hence fixing $1\leq i\leq e$ we obtain
	\begin{align*}
	\nonumber\norm{P_iX^E}&\leq |a|\cdot\left(\norm{P_i\circ\varphi_r|E^u_A}+2\Big(\norm{P_i\circ\varphi_r|E^s_A}+\norm{P_i\circ d_{\wtf^{-1}p}S_r}+\phi(r)\Big)\right)\\ 
	&\leq \frac{\norm{P_i\circ\varphi_r|E^u_A}}{\lambda-2\phi(r)}\Big(1+2\left\{\frac{\norm{P_i\circ\varphi_r|E^s_A}+\norm{P_i\circ d_{\wtf^{-1}p}S_r}+\phi(r)}{\norm{P_i\circ\varphi_r|E^u_A}}\right\}\Big)\\               
	\norm{P_iX^E}&\geq |a|\cdot\left(\norm{P_i\circ\varphi_r|E^u_A}-2\Big(\norm{\varphi_r|E^{s}_A}+\norm{P_i\circ d_{\wtf^{-1}p}S}+\phi(r)\Big)\right)\\
	&\leq \frac{\norm{P_i\circ\varphi_r|E^u_A}}{\lambda+2\phi(r)}\Big(1-2\left\{\frac{\norm{P_i\circ\varphi_r|E^{s}_A}+\norm{P_i\circ d_{\wtf^{-1}p}S}+\phi(r)}{\norm{P_i\circ\varphi_r|E^u_A}}\right\}\Big).
	\end{align*}
	Condition \textbf{A-2} implies that by taking $\phi(r)$ small, the terms in braces converge to zero as $r\to+\oo$ and thus we deduce the existence of $\varpi_1(r)>0$ with $\varpi_1(r)\xrightarrow[r\to+\oo]{}0$ such that
	\begin{equation}\label{eq.despi}
	(1-\varpi_1(r))\frac{\norm{P_i\circ\varphi_r|E^u_A}}{\lambda}\leq \norm{P_iX^E}\leq (1+\varpi_1(r))\frac{\norm{P_i\circ\varphi_r|E^u_A}}{\lambda}.
	\end{equation}
	For $1\leq i, j\leq e$ we then have
	\[
	\frac{\norm{P_iX^E}}{\norm{P_jX^E}}\geq K(r)\cdot \frac{1-\varpi_1(r)}{1+\varpi_1(r)}
	\]
	and thus
	\begin{equation}\label{eq.Pij}
	K(r)\cdot \frac{1-\varpi_1(r)}{1+\varpi_1(r)}\leq\frac{\norm{P_iX^E}}{\norm{P_jX^E}}\leq K(r)^{-1}\cdot \frac{1+\varpi_1(r)}{1-\varpi_1(r)}.
	\end{equation}
	Since $K(r)\xrightarrow[r\to+\oo]{}1$, the above gives the first inequality in the lemma.

	Similarly, by \eqref{eq.elv}, 
	\begin{align*}
	\norm{X^E}&\leq |a|\cdot\left(\norm{\varphi_r|E^u_A}+2\Big(\norm{\varphi_r|E^s_A}+\norm{d_{\wtf^{-1}p}S_r}+\phi(r)\Big)\right)\\
	&\leq\frac{\norm{\varphi|E^u_A}}{\lambda}\cdot\frac{1+2\Big\{\frac{\norm{\varphi|E^{s}_A}+\norm{dS}+2\phi(r)}{\norm{\varphi|E^u_A}}\Big\}}{1-\frac{2\phi(r)}{\lambda}}\\
	\norm{X^E}&\geq |a|\cdot\left(\norm{\varphi_r|E^u_A}-2\Big(\norm{\varphi_r|E^{s}_A}+\norm{d_{\wtf^{-1}p}S}+\phi(r)\Big)\right)\\
	&\geq\frac{\norm{\varphi|E^u_A}}{\lambda}\cdot\frac{1-2\Big\{\frac{\norm{\varphi|E^{s}_A}+\norm{dS}+2\phi(r)}{\norm{\varphi|E^u_A}}\Big\}}{1+\frac{2\phi(r)}{\lambda}}
	\end{align*}
	Condition \textbf{A-1} implies that the quantity $\displaystyle{\frac{\norm{\varphi|E^{s}_A}+\norm{dS}+2\phi(r)}{\norm{\varphi|E^u_A}}}$ converges to $1$ (assuming $\phi(r)\leq 1$), thus there exists $\varpi_2(r)>0$ converging to zero as $r\to+\oo$ such that
	\begin{equation*}
	(1-\varpi_2(r))\frac{\norm{\varphi_r|E^u_A}}{\lambda}\leq \norm{X^E}\leq (1+\varpi_2(r))\frac{\norm{\varphi_r|E^u_A}}{\lambda},
	\end{equation*}
	which in turn implies
	\begin{equation}
	1-\varpi_3(r)\leq \norm{X}\leq 1+\varpi_3(r)
	\end{equation}
	where $\varpi_3(r)=(1+\varpi_2(r))\frac{\norm{\varphi_r|E^u_A}}{\lambda}$. Note that $\varpi_3(r)\xrightarrow[r\to+\oo]{}0$ by condition \textbf{A-1}.
	Combining this with \eqref{eq.despi} we finally get
	\begin{equation}\label{eq.desfinal}
	\frac{1-\varpi_1(r)}{1+\varpi_3(r)}\cdot\frac{\norm{P_i\circ\varphi_r|E^u_A}}{\lambda}\leq \frac{\norm{P_iX^E}}{\norm{X}}\leq \frac{1+\varpi_1(r)}{1-\varpi_3(r)}\cdot\frac{\norm{P_i\circ\varphi_r|E^u_A}}{\lambda}.
	\end{equation}
	
	Inequalities \eqref{eq.Pij} and \eqref{eq.desfinal} imply the conclusion of the Lemma, by defining $\varpi(r)$ appropriately. 
	
\end{proof}

\noindent\textbf{Convention: }From now on $r$ is taken sufficiently large and $\phi(r)$ small to verify the hypotheses of the previous lemma. To avoid cluttering the notation, the quantity $\varpi(r)$ will be subsequently re-defined to guarantee additional conditions, but always maintaining $\varpi(r)\xrightarrow[r\To +\oo]{}0$.

\esp

\begin{definition}
	An admissible curve for the map $\wtf$\ is a curve $\gamma: [0,2\pi]\rightarrow M$\ tangent to $E^{u}_{\wtf}$\ such that for some $i\in{1,\ldots,e}$ it holds $\left|P_i\circ\frac{d\gamma}{dt}(t)\right|=1.$ 
\end{definition}

\begin{remark}\label{curvastotales}
	If $\ga$ is an admissible curve with $\left|P_i\circ\frac{d\gamma}{dt}(t)\right|=1$ then its projection on $\Tor_i$ makes exactly one turn. By the lemma above $\left|P_j\circ\frac{d\gamma}{dt}(t)\right|\approx 1$ for any other $1\leq j\leq e$, and the error in this approximation converges to zero as $r\To +\oo$. We deduce that the projection of $\ga$ into $\Tor_j$ completes almost a turn, with an error that can be taken arbitrarily small as $r$ increases.  
\end{remark}

A consequence of the lemma above is the following.

\esp

\begin{corollary}\label{longcurvas}
	There exists $\varpi(r)$ with $\lim_{r\to+\oo}\varpi(r)=0$ such that for any admissible curve $\gamma$ its length $Leb(\gamma)$ is bounded as
	
	\begin{equation*}
	2\pi(1-\varpi(r))\cdot\frac{\lambda}{\norm{P_i\circ\varphi|E^u_A}}\leq Leb(\gamma)\leq 2\pi(1+\varpi(r))\cdot\frac{\lambda}{\norm{P_i\circ\varphi|E^u_A}}
	\end{equation*}
	In particular, if $\gamma,\gamma'$ are admissible curves, for sufficiently large $r$ it holds
	\[(1-\varpi(r))\cdot  Leb(\gamma)\le  Leb(\gamma')\le (1+\varpi(r))\cdot Leb(\gamma).\]
\end{corollary}

\begin{proof}
	Consider an admissible curve $\gamma$ with $\left|P_i\circ\frac{d\gamma}{dt}(t)\right|=1$. For every $t$ the Lemma above gives  
	\[
	\frac{1}{1+\varpi(r)}\cdot\frac{\lambda}{\norm{P_i\circ\varphi_r|E^u_A}}\leq\norm{\frac{d\gamma}{dt}(t)}\leq \frac{1}{1-\varpi(r)}\cdot\frac{\lambda}{\norm{P_i\circ\varphi_r|E^u_A}}
	\]
	and thus $\frac{2\pi}{1+\varpi(r)}\cdot\frac{\lambda}{\norm{P_i\circ\varphi_r|E^u_A}}\leq Leb(\gamma)\leq \frac{2\pi}{1-\varpi(r)}\cdot\frac{\lambda}{\norm{P_i\circ\varphi_r|E^u_A}}$. Re-defining $\varpi(r)$ we get both results, by using the last part of \textbf{A-2}.
\end{proof}

\esp

\subsection{Adapted fields} We turn our attention to some special class of vector fields parallel to $E^c_{\wtf}$. Recall that $\theta$ is the H\"older exponent of the center bundles (which can be taken uniform in a small neighborhood of $f$), and $l\in\mathbb{N}$ is given in condition \textbf{A-1}. Let 

\begin{equation}\label{eq.Co}
C_0=\frac{1}{\lambda^{\theta(1-\frac{1}{2l})}}
\end{equation}

\begin{definition}
	An adapted field $(\gamma,X)$ for $\wtf$\ consists of an admissible curve $\gamma$ and a unit vector field $X$ along $\gamma$ satisfying the following.
	\begin{enumerate}
		\item $X$\ is tangent to $E_{\wtf}^c$.
		\item $X$ is $(C_0,\theta)$-H\"older along $\gamma$. This means that if $d_{\gamma}$ denotes the intrinsic distance in $\gamma$, it holds 
		\[p,p'\in\gamma\Rightarrow\norm{X_p-X_{p'}}\leq C_0\cdot d_{\gamma}(p,p')^{\theta}.
		\] 
	\end{enumerate}
\end{definition}

Even though the length of the admissible curves is rather large, the bound on the constant $C_0$ makes the variation of any adapted field to be very small. Using Corollary \ref{longcurvas} one establishes the following without any trouble (recall that $l$ is defined in condition \textbf{A-1}). 

\esp

\begin{corollary}\label{variacionX}
	If $(\gamma,X)$ is an adapted field for $\wtf$ and $r$ is sufficiently large, then for every $p,p'\in \gamma$ it holds
	\[
	\norm{X_p-X_{p'}}<\left(\frac{\lambda^{\frac{1}{2l}}}{\norm{\varphi}}\right)^{\theta}.
	\]
\end{corollary}

In particular the variation of $X$ converges to zero as $r\To+\oo$, due to $\textbf{A-1}$.

\esp

We will now show that the set of adapted fields has certain $\wtf$-invariance. Let $(\ga,X)$ be an adapted field. If $k\geq0$ we can write the curve $\wtf^k\circ\gamma$\ as
\begin{equation}\label{concauf} 
\wtf^k\circ\gamma=\gamma_1^k\ast\cdots\gamma_{N_k}^k\ast \gamma_{N_k+1}^k
\end{equation}
where $N_k=N_k(\gamma)$ is an integer, $\gamma^k_j$ are admissible curves for $j=1,\ldots, N_k$\  and $\gamma^k_{N_k+1}$\ is a segment of an admissible curve\footnote{ Here $\ast$ denotes concatenation of paths.}. Even more, by the Remark \ref{curvastotales} the curves $\ga^{k}_j$ can be taken so that they make a complete turn around each of the coordinate torii $\Tor_i$. 

If $X$ is a vector field we denote $\wtf_{\ast}X=d\wtf\circ X\circ \wtf^{-1}$, the push-forward by $\wtf$.  Now let 
\begin{equation}\label{Yk}
Y^k:=\frac{(\wtf^k)_{\ast}X}{\norm{(\wtf^k)_{\ast}X}};
\end{equation}
we have the following.
\begin{lemma}\label{adapted}
	For sufficiently large $r$ it holds the following. If $(\ga,X)$ is an adapted field and $k\geq 0$ then for every  $1\leq j \leq N_k$ the pair $(\ga^k_j,Y^k|\ga_j^k)$ is an adapted field.
\end{lemma}

\begin{proof}
	This is similar to Lemma 2 in \cite{NUHD}. The proof is given below for completeness.
	Arguing by induction it is enough to show the claim for $k=1$. Denote $Y:=Y^1$ and observe
	\begin{equation*}
	\forall p,p'\in M\quad\norm{d_p\wtf(X_p) -d_{p}\wtf(X_{p'})}\leq 2\norm{dS}\cdot\norm{X_p-X_{p'}}\leq 2\norm{dS}\cdot C_0\cdot d_{\gamma}(p,p')^{\theta}
	\end{equation*}
	and (choosing $\phi(r)$ small enough),
	\begin{equation*}
	\norm{d_p\wtf(X_{p'})-d_{p'}\wtf(X_{p'})}\leq
	\begin{dcases*}
	2\norm{d^2S}\cdot d_{\gamma}(p,p')\leq 3\norm{d^2S}\cdot d_{\gamma}(p,p')^{\theta} & if $d_{\gamma}(p,p')< 1$\\
	3\norm{dS}\leq 3\norm{dS}\cdot d_{\gamma}(p,p')^{\theta} & if $d_{\gamma}(p,p')\geq 1$
	\end{dcases*}
	\end{equation*}
	hence $\forall p,p'\in M$,
	\[
	\norm{d_p\wtf(X_p)-d_{p'}\wtf(X_{p'})}\leq \Big(3\max\{\norm{dS},\norm{d^2S}\}+2\norm{dS}C_X\Big)\cdot d_{\gamma}(p,p')^{\theta}.
	\]
	On the other hand, by triangular inequality, for $p,p'\in \ga^1_j$
	\begin{align*}
	&\norm{Y_p-Y_{p'}}=\frac{1}{\norm{\wtf_{\ast}X_p}\norm{\wtf_{\ast}X_{p'}}}\Bigg\|\norm{\wtf_{\ast}X_{p'}}\wtf_{\ast}X_p -\norm{\wtf_{\ast}X_p}\wtf_{\ast}X_{p'}\Bigg\|\\
	&\leq\frac{1}{\norm{\wtf_{\ast}X_p}\norm{\wtf_{\ast}X_{p'}}}\left(\Bigg\|\norm{\wtf_{\ast}X_{p'}}\wtf_{\ast}X_p-\norm{\wtf_{\ast}X_{p'}}\wtf_{\ast}X_{p'}\Bigg\|+\Bigg\|\norm{\wtf_{\ast}X_{p'}}\wtf_{\ast}X_{p'}-\norm{\wtf_{\ast}X_p}\wtf_{\ast}X_{p'}\Bigg\|\right)\\
	&\leq\frac{2}{\norm{\wtf_{\ast}X_p}}\norm{\wtf_{\ast}X_p-\wtf_{\ast}X_{p'}}=\frac{2}{\norm{\wtf_{\ast}X_p}}\norm{d_{\wtf^{-1}p}\wtf(X_{\wtf^{-1}p})-d_{\wtf^{-1}p'}\wtf(X_{\wtf^{-1}p'})}
	\end{align*}
	Putting both inequalities together we deduce for $p,p'\in \ga^1_j$
	\begin{align*}
	\norm{Y_p-Y_{p'}}&\leq \frac{2}{\norm{\wtf_{\ast}X_p}}\Big(3\max\{\norm{dS},\norm{d^2S}\}+2\norm{dS}C_0\Big)\cdot d_{\gamma}(\wtf^{-1}p,\wtf^{-1}p')^{\theta}\\
	&\leq 2\norm{dS^{-1}}\Big(3\max\{\norm{dS},\norm{d^2S}\}+2\norm{dS}C_0\Big)\frac{1}{\lambda^{\theta}}d_{\gamma}(p,p')^{\theta}\\
	&=\Big(\frac{6\max\{\norm{dS^{-1}}\norm{dS},\norm{dS^{-1}}\norm{d^2S}\}}{\lambda^{\theta}}+\frac{4\norm{dS^{-1}}\norm{dS}}{\lambda^{\theta}}C_0\Big)d_{\gamma}(p,p')^{\theta}\\
	&<\Big(\frac{C_0}{2}+\frac{C_0}{2}\Big)d_{\gamma}(p,p')^{\theta}=C_0\cdot d_{\gamma}(p,p')^{\theta}
	\end{align*}
	if $r$ is sufficiently large and $\theta$ close to 1, by \textbf{A-1}.
\end{proof}

\section{Positivity of the center exponents}

To establish the existence of positive Lyapunov exponents in the $E^c_{\wtf}$ directions we will study the quantities 
\begin{equation}
I_n(\gamma,X):=\frac{1}{|\gamma|}\int_{\gamma} \log\norm{d_p\wtf^n(X_p)}d\gamma\qquad n\in\mathbb{N} 	
\end{equation}
where $(\gamma,X)$ is an admissible curve and $d\gamma$ denotes the intrinsic Lebesgue measure in $\gamma$. 

By our choice of $\phi(r)$, the center direction $E_{\wtf}^c$ of $\wtf$ is close to $E^c_f=V$, and in particular there exists a bundle isomorphism $\displaystyle{T_{\wtf}:E_{\wtf}^c\rightarrow V}$ that can be chosen to depend continuously with respect to the map $\wtf$ (in particular $\displaystyle{T_{\wtf}\xrightarrow[\wtf\To f]{}Id:V\rightarrow V}$). We will write $\displaystyle{E_{i,\wtf}^c=T_{\wtf}^{-1}(\Real^2_i)}$  and observe that $\displaystyle{E_{\wtf}^c=\oplus_{i=1}^{e} E_{i,\wtf}^c}$ although this decomposition is not $d\wtf$ invariant in general.  The cone fields $\displaystyle{\Delta_{i,r}, 1\leq i\leq e}$ associated $S_r$ are extended (with the same nomenclature) to cone fields in $M$,  and by using  $T_{\wtf}$ we obtain cone fields $\displaystyle{\Delta_{i,\wtf}\subset E_{i,\wtf}^c}$.

We write $J^{u}_{\wtf^{-k}}:=\det|d\wtf^{k}|E^{u}_{\wtf}|$ for the unstable Jacobian of $\wtf$, and  recall the following classical Lemma\footnote{cf.\@ Lemma 8 and Corollary 6 in \cite{NUHD} : observe that by Corollary \ref{longcurvas} the length of the admissible curves is bounded.}.

\esp

\begin{lemma}[Distortion bounds]\label{distortion}
	There exists a constant $\mathcal{E}=\mathcal{E}_r$ such that for every admissible curve $\gamma$ it holds
	\[
	\forall p,p'\in \ga,\quad  \frac{1}{\mathcal{E}}\leq \frac{J^u_{\wtf^{-k}}(p)}{J^u_{\wtf^{-k}}(p')} \leq \mathcal{E}.
	\]
	Therefore for every measurable set $A\subset \gamma$ and $k\geq 0$ we have
	\[
	\frac1{\mathcal{E}}\cdot\frac{Leb(A)}{Leb(\ga)}\leq \frac{Leb(\wtf^{-k}A)}{Leb(\wtf^{-k}\ga)}\leq \mathcal{E}\cdot\frac{Leb(A)}{Leb(\ga)}.
	\]
	Moreover, if $r$ sufficiently large and $\phi(r)$ small it holds $\mathcal{E}(r)\approx 1.$
\end{lemma}

\begin{remark}
	The last part follows by Lemma \ref{accioninestable} and continuous dependence of the partially hyperbolic splitting on the map.
\end{remark}

\esp

The fundamental Proposition is the following.

\begin{proposition}\label{ergog}
	Suppose that there exist $C>0$, a full Lebesgue measure set $M_0$ and continuous cones $\Delta_{i,\wtf}\subset E_{i,\wtf}^c, 1\leq i\leq e$ with the following properties.
	\begin{enumerate}
		\item $M_0$ is saturated by the foliation $\F^{u}_{\wtf}$. That is, $M_0$ consists of complete unstable leaves of $\F^{u}_{\wtf}$.
		\item Given an adapted field $(\ga,X)$ with $\gamma\subset M_0$  and satisfying $X_p\in \Delta_{i,\wtf}(p)$ for all $p\in \gamma$ it holds
		\[
		\liminf_{n\To\oo}\frac{I_n(\gamma,X)}{n}\geq C
		\]
	\end{enumerate}
	Then for Lebesgue almost every point, the sum of the multiplicity of center exponents bigger than equal $C/\mathcal{E}$ is at least $e$.
\end{proposition}
This generalizes Proposition 5 in \cite{NUHD} to our setting, and even improves the lower bound on the exponents. 

\begin{proof}
    Fix $1\leq i\leq e$ and consider the set $B_i=\{p:\forall v\in E_{i,\wtf}^c\setminus\{0\},\ \chi(p,v)<\frac{C}{\mathcal{E}}\}$. Write $B_i=\cup_{j=1}^{\oo} B_{i,j}$ where $B_{i,j}=\{p:\forall v\in E_{i,\wtf}^c\setminus\{0\}, \chi(p,v)<\frac{C}{\mathcal{E}}(1-\frac{1}{j})\}$ and assume for the sake of contradiction that $\mu(B_i)>0$; then for some $j$ it holds $\mu(B_{i,j})>0$.  By absolute continuity of the foliation $\F^{u}_{\wtf}$ (see for example \cite{PesinLect}) and the hypothesis on $M_0$ there exists an interval $L$ inside an unstable leaf with $L\subset M_0$ and so that $L\cap B_{i,j}$ has positive intrinsic Lebesgue measure in it. Take a density point $b\in B_{i,j}\cap L$ for the measure in $L$.
	
	Given an admissible curve $\gamma$ we call $p_{\ga}\in \gamma$ its center if $Leb([\gamma(0),p_{\ga}])=\frac{Leb(\gamma)}{2}$, where $[\gamma(0),p_{\ga}]$ denotes the (oriented) interval inside $\gamma$. For $\epsilon>0$ small and some $k$ large to be specified later consider $\ga^{\ep}:[-\ep,\ep]\rightarrow M$,
	\[	
	\ga^{\ep}(t)=\wtf^{-k}\circ \be_{k}(t)
	\]
	where $\be_{k}$\ is the admissible curve with center $\wtf^k(b)$. Note that $Leb(\ga^{\ep})$\ decreases with $k$, and even though this curve is not necessarily symmetric with respect to $b$, the ratio of the length of the intervals $[\ga^{\ep}(-\ep),b],[b,\ga^{\ep}(\epsilon)]$ is close to one, by Lemma \ref{distortion} and almost conformality of $\wtf$ on its unstable foliation. Thus by Lebesgue's differentiation theorem,  $\displaystyle{\frac{Leb(\gamma^\epsilon\cap B_{i,j})}{Leb(\gamma^\epsilon)}\xrightarrow[|\gamma^{\epsilon}|\To 0]{}1}$, which implies if $\gamma^\epsilon$ small enough (or equivalently $k$ sufficiently large) 
	\[\frac{Leb(\ga^{\ep}\cap B^c_{i,j})}{Leb(\ga^{\ep})}<\frac{C}{j\cdot \mathcal{E}\cdot\norm{d\wtf|E^c_{\wtf}}}.\]
	Note also that for every point $p\in \be_{k}$ one has 
	\begin{equation*}
	J^u_{\wtf^{-k}}(p)\ge \frac{Leb(\gamma^\epsilon)}{\mathcal{E}\cdot Leb(\be_{k})}
	\end{equation*}
	Take $v\in T_bM$ such that $d_b\wtf^k(v)\in \Delta_{i,\wtf}(b)$. We claim that $\chi(b,v)\geq \frac{C}{\mathcal{E}}(1-\frac{1}{j})$, which gives a contradiction since $b\in B_{i,j}$.
	
	By contradiction, suppose that $v$ is as before and extend it to a continuous vector field $X$ over $\wtf^{-k}\circ \be_{k}(t)$ with the property that $(\be_{k},\wtf_{\ast}^kX)$  is an adapted field satisfying $\wtf_{\ast}^kX\in \Delta_{i,\wtf}$ (this is possible since $\Delta_{i,\wtf}$ is continuous). Consider for $p\in \gamma^\epsilon$ the quantity
	\[\chi(p)=\limsup_n \frac1n \log\|d\wtf^n \circ X\circ \wtf^k(p)\|.\]
	
	We compute, using (reverse) Fatou's Lemma
	\begin{equation*}\begin{split}
	\int_{\ga^{\ep}}\chi d\ga^{\ep}=&\int_{\wtf^k\ga^{\ep}}\chi\circ \wtf^{-k} J^u_{\wtf^{-k}}d(\wtf^k\ga^{\ep})\ge \frac{Leb(\gamma^\epsilon)}{\mathcal{E}\cdot Leb(\be_{k})} \int_{\be_{k}} \chi\circ \wtf^{-k} d(\be_{k})\\
	&\ge \frac{Leb(\gamma^\epsilon)}{\mathcal{E}\cdot Leb(\be_{k})}\cdot C\cdot Leb(\be_{k})=\frac{C}{\mathcal{E}}\cdot Leb(\gamma^\epsilon)
	\end{split}
	\end{equation*}
	On the other hand, since $\chi(p)<\frac{C}{\mathcal{E}}(1-\frac{1}{j})$ for $p\in B_{i,j}$,
	\begin{align*}
	&\int_{\ga^{\ep}}\chi d\ga^{\ep}=\int_{\ga^{\ep}\cap B_{i,j}}\chi d\ga^{\ep}+\int_{\ga^{\ep}\cap B_{i,j}^c}\chi d\ga^{\ep}< \frac{C}{E}(1-\frac{1}{j})Leb(\ga^{\ep}\cap B_{i,j})+\norm{d\wtf|E^c_{\wtf}}Leb(\ga^{\ep}\cap B_{i,j}^c)\\
	&<\frac{C}{\mathcal{E}}Leb(\ga^{\ep})
	\end{align*}
	which is absurd. We have thus proved that $\mu(B_i)=0$.
	
	For a center Lyapunov exponent $\chi_i^c$ denote $\mathrm{mult}_i$ its multiplicity; then 
	\begin{align*}
	&\mu(\{p\in M:\sum_{\chi^c_i(p)\geq C/\mathcal{E}}\mathrm{mult}_i(p)\geq e\})=\mu((\cup_{i=1}^e B_i)^c)=1.	
	\end{align*}
\end{proof}

\begin{remark}
	If the reader wants to consider the more general case (cf. Remark \ref{general}) $S_{i,r}:\Toro^{d_i}\rightarrow\Toro^{d_i}, d_i\geq 2$ where $\Delta_{i,r}$ is a cone centered around a subspace in $W_i\subset E_{i,\wtf}^c$, the set $B_i$ in the previous proof has to be replaced by 
	\[
	B_i=\{p\in M:\sum_{\chi^c_i(p)< C/\mathcal{E}}\mathrm{mult}_i(p)>\dim E_{i,\wtf}^c-\dim W_i\}
	\]
	and the argument follows along the same lines to give that for almost every $p$ there exists at least $\sum_{i}^{e}\dim W_i$ center positive exponents larger that $\frac{C}{\mathcal{E}}$.
\end{remark}	

\subsection{Study of the integral} Recall \eqref{concauf} and the definition of $Y^k$ given in \eqref{Yk}. One can write 
\[I_n(\ga,X)=\sum_{k=0}^{n-1}\frac{1}{|\ga|}\int_{\ga} \log\norm{d_{\wtf^kp}\wtf(Y^k\circ \wtf^k(p))}d\ga= \sum_{k=0}^{n-1}\frac{1}{|\ga|}\int_{f^k\ga} \log\norm{d_{p}\wtf Y^k}J^{u}_{\wtf^{-k}}d\ga,\] 
therefore by \eqref{concauf}
\begin{equation}
I_n(\ga,X)= \sum_{k=0}^{n-1}\bigg(R_k+ 
\sum_{j=0}^{N_k}
\frac{1}{|\ga|} \int_{\ga_j^k} \log{\norm{d_{p}\wtf(Y^k)}}J^u_{\wtf^{-k}}d\ga_j^k
\bigg),	
\end{equation}
with $R_k= \frac{1}{|\ga|} \int_{\ga^k_{N_k+1}} \log{\norm{d_{p}\wtf(Y^k)}}J^u_{\wtf^{-k}} d\ga_{N_k+1}^k$ (observe that by  Lemma \ref{adapted}, for every $0\leq k<n,0\leq j\leq N_k$ the pair $(\gamma_{j}^k,Y^k)$ is an adapted field). It will be convenient to introduce the following notation.

\begin{definition}
	If $(\gamma,X)$ is an adapted field then
	\[
	I(\gamma,X):=\frac{1}{|\gamma|}\int_{\gamma}\log\norm{d_p\wtf(X_p)}d\gamma.
	\]
\end{definition}

Using Corollary \ref{longcurvas} we deduce that for every $j, k$
\begin{equation}\label{comparacurva}
\frac{|\ga_j^k|}{|\ga|}\ge 1-\varpi(r) \quad\text{and}\quad \frac{|\ga_{N_k+1}^k|}{|\ga|}\le 1+\varpi(r),	
\end{equation}
thus (again, re-defining $\varpi$)
\begin{equation*}
I_n(\ga,X)\geq  \sum_{k=0}^{n-1}\bigg(R_k+ 
\left(1-\varpi(r)\right) \sum_{j=0}^{N_k}
\min_{\ga_j^k} (J^u_{\wtf^{-k}}) \cdot I(\ga_j^k, Y^k) )
\bigg).
\end{equation*}

\begin{lemma}
	It holds
	\[
	\lim_{n\rightarrow \infty}\frac1{n} \sum_{k=0}^{n-1} |R_k|=0.
	\]
\end{lemma}
\begin{proof}
	We compute
	\begin{align*} 
	|R_k|&\leq (1+\varpi(r))\cdot\max_{\ga_{N_k+1}^k} | J^u_{\wtf^{-k}}|\cdot  \log{\norm{d \wtf|E^c_{\wtf}}},
	\end{align*}
	This implies that $|R_k|$ converges to zero as $k$ goes to infinity, hence so does its average.
\end{proof}

It follows,

\begin{equation}
\liminf_{n\To\oo}\frac{I_n(\ga,X)}{n}\geq  (1-\varpi(r))\liminf_{n\To\oo}\frac{1}{n}\sum_{k=0}^{n-1} 
\sum_{j=0}^{N_k}\min_{\ga_j^k} (J^u_{\wtf^{-k}}) \cdot I(\ga_j^k, Y^k)).
\end{equation}

\subsection{Good and bad vector fields} We will now check that for each $i$ the cone $\Delta_{i,\wtf}\subset E_{i,\wtf}^c$ induced by $\Delta_i$ has the property that if $(\ga,X)$ is an adapted field with $X\in \Delta_{i,\wtf}$ then the right hand side term in the previous inequality is positive, thus showing (thanks to Proposition \ref{ergog}) the existence of $e$ positive Lyapunov exponents. Only now the more specific aspects of the dynamics $S_r$ (the existence of the cones) enters in consideration.

\begin{proposition} \label{fonda} There exists a positive constant $Q=Q(r)>0$ such that for every $k\geq 0$, for every admissible field $(\gamma,X)$ satisfying $X\in \Delta_{i,\wtf}$, it holds
	\[	\sum_{k=0}^{N_k}
	\min_{\ga_j^k}(J^u_{\wtf^{-k}} \cdot I(\ga_j^k, Y^k) )\ge Q.\]
\end{proposition}

The above proposition will be proven through a series of Lemmas. A vector $X_p\in E_{\wtf}^c(p)$ can be written uniquely as
\[
X_p=X_{p,1}+\cdots X_{p,e}\quad X_{p,i}\in E_{i,\wtf}^c(p).
\] 
Recall that in $V=E^c_f$ we are using the $\ell^{\oo}$ norm in associated to the decomposition $V=\bigoplus_{i=1}^e\Real^2_i$. Using $T_{\wtf}$, we induce the corresponding  $\ell^{\oo}$ norm in $E_{\wtf}^c$ associated with the decomposition $E_{\wtf}^c=\bigoplus_{i=1}^e E_{i,\wtf}^c(p)$, i.e.\@ $\norm{X_p}=\max_{i=1,\cdots, e} \norm{X_{p,i}}$.  We say that $X_{p,i}$ is a leading component of $X_p$ if $\norm{X_{p,i}}=\norm{X_p}$. By Corollary \ref{variacionX} we deduce the following.

\begin{lemma}
	For $r$ sufficiently large it holds that if $(\ga,X)$ is an adapted field for $\wtf$ and for some $p\in \ga$ the vector $X_{p,i}$ is a leading component of $X_p$, then 
	\[
	p'\in \ga\Rightarrow \norm{X_{p',i}}\geq 1-\varpi(r).
	\]
\end{lemma}

\begin{definition} Let $(\ga,X)$ be an adapted vector field  for $\wtf$.
	\begin{enumerate}
		\item We say that $X_{\cdot,i}$ is a leading component of $X$ if there exists $p\in \ga$ such that $X_{p,i}$ is a leading component of $X_p$.
		\item $X$ is called good if it has some leading component $X_{\cdot,i}$ satisfying $\forall p\in \ga$
		\[
		X_{p,i}\in \Delta_{i,\wtf}(p).
		\]
		Otherwise it is called bad.	
	\end{enumerate}	
\end{definition}

\esp

The functions $\beta_i(r)>\zeta_i(r)$ are specified in condition \textbf{S-2}.

\begin{lemma}\label{relacionEga}
	For $r$ sufficiently large and $\phi(r)$ small it holds for every adapted field $(\ga,X)$ for $\wtf$ with leading component $X_{\cdot,i}$,
	\begin{enumerate}
		\item $I(\ga,X)\geq (1-10^{-23})\log \zeta_i(r)$.
		\item If moreover $(\ga,X)$ is good, then $I(\ga,X)\geq (2-10^{-23})\pi\log \beta_i(r)$.
	\end{enumerate}
\end{lemma}

\esp

\begin{proof}
	Denote by $\pi_i: E^c_{\wtf}\rightarrow E_{i,\wtf}^c$ the projection. Since the bundles $\Tor_i\subset V$ are $df$-invariant, for $\phi(r)$ sufficiently small we have that for $i\neq j$ $\displaystyle{\norm{\pi_j\circ d\wtf\circ \pi_i}\approx 0}$, and $\inf_p m(\pi_i \circ d_p\wtf\circ \pi_i )\approx \inf_p m(d_pf|\Real^2_i)= \zeta_i(r)$. 
	One then has 
	\[
	\norm{d_p\wtf(X_p)}\geq \norm{\pi_i\circ d_p\wtf(X_p)}\approx \norm{\pi_i\circ d_p\wtf(X_{p, i})}\geq a(r)\zeta_i(r) 
	\]
	where $a(r)\xrightarrow[\phi(r)\To 0]{}1$; to conclude the first part it is enough then to choose $\phi(r)$ is sufficiently small so that $a(r)\geq \frac{1}{\sqrt[10^{23}]{\zeta_i(r)}}$. Arguing similarly we also obtain that if $(\ga,X)$ is good then for $p\not\in \mathcal{C}_{i}$, $\norm{d_p\wtf(X_p)}\geq  a'(r)\beta_i(r)$ with $a'(r)\xrightarrow[\phi(r)\To 0]{}1$; note that the projection of $\ga$ on $\Tor_i$ may be doing slightly more than one turn (cf. Remark \ref{curvastotales}), but since the length of this additional part converges to zero as $r\To +\oo$, we can assume as a worse case scenario situation  that the additional part is contained in $\mathcal{C}_{i,r}$, and hence  
	\begin{align*}
	I(\ga,X)&\geq	(2\pi-2l(\mathcal{C}_i))(a'(r))\cdot \log\beta_i(r)-2l(\mathcal{C}_i)(1-10^{-23})\cdot \log\zeta_i(r)\\
	&\geq (2-10^{-23})\pi\log \beta_i(r) 
	\end{align*}
	if $r$ is sufficiently large (hence $l(\mathcal{C}_i)$ small), and $\phi(r)$ small. 
\end{proof}

Given an adapted vector field $(\ga,X)$ and $k\geq 0$ we define
\begin{align}
&G_k=G_k(\ga,X)=\{(\ga^k_j,Y^k|\ga_j^k)\text{ good adapted field}: 1\leq j \leq N_k\}\\
&B_k=B_k(\ga,X)=\{(\ga^k_j,Y^k|\ga_j^k)\text{ bad adapted field}: 1\leq j \leq N_k\}	
\end{align}
cf. Lemma \ref{adapted}. Note $\# G_k+\# B_k= N_k$.

\begin{lemma}\label{numberg}
	For every adapted field $(\ga,X)$ and every positive integer $k\geq 0$, it holds
	\[
	1-10^{-23}\leq\sum_{j\in G_k} \min_{\ga_j^k}J^u_{\wtf^{-k}}+ \sum_{j\in B_k} \max_{\ga_j^k}J^u_{\wtf^{-k}}\le 1+10^{-23}.
	\]
	provided that $\phi(r)$ is small and $r$ large.
\end{lemma}
\begin{proof}
	By Lemma \ref{distortion} and \ref{comparacurva}
	\begin{align*}
	1&=\frac{1}{|\ga|}\int_{\ga}d\ga=\frac{1}{|\ga|}\sum_{k=1}^{N_k+1}\int_{\ga_j^k}  J^u_{\wtf^{-k}}d\ga_j^k
	\geq 
	\left(\sum_{j\in G_k} \frac{|\ga_j^k|}{|\gamma|} \min_{\ga_j^k}J^u_{\wtf^{-k}}+ \sum_{j\in B_k\cup\{N_k+1\}} \frac{|\ga_j^k|}{|\gamma|}\min_{\ga_j^k}J^u_{\wtf^{-k}}\right)\\	
	&\geq (1-\varpi(r))
	\left(\sum_{j\in G_k}  \min_{\ga_j^k}J^u_{\wtf^{-k}}+ \frac{1}{\mathcal{E}}\sum_{j\in B_k} \max_{\ga_j^k}J^u_{\wtf^{-k}}\right).
	\end{align*}
	Using that $\mathcal{E}\xrightarrow[\phi(r)\To 0]{} 1$ the first inequality follows. The second one is similar.
\end{proof}

The next lemma takes care of the transitions between good and bad vector fields. Recall the definition of $R,l(\mathcal{C}_{i,r})$ given in \textbf{S-1} and denote
\begin{equation}
\rho(r):=e\cdot \max_{1\leq i\leq e}\frac{l(\mathcal{C}_{i,r})}{2\pi}.
\end{equation}

\begin{lemma}\label{tripforg} For sufficiently large $r$ and correspondingly small $\phi(r)$ the following holds.
	
	\begin{enumerate}
		\item[(a)] 	If $(\ga,X)$\ is a good adapted field then there exists a relatively open set $\ga_g\subset \ga$ of length $|\ga_g|\geq (1-\rho(r))|\ga|$ such that 
		if $\wtf^{-1}\ga^1_j\subset \ga_g$ then  $(\ga^1_j,\frac{\wtf_{\ast}X}{\norm{\wtf_{\ast}X}})$\ is good.
		
		\item[(b)] If $(\ga,X)$\ is a bad adapted field then there exists a relatively open set $\ga_{bg}\subset \ga$ of length $|\ga_g|\geq 0.99R|\ga|$ such that
		if $\wtf^{-1}\ga^1_j\subset \ga_{bg}$ then$(\ga^1_j,\frac{\wtf_{\ast}X}{\norm{\wtf_{\ast}X}})$ is good. 
	\end{enumerate}
\end{lemma}

\begin{proof}
	We first deal with the case $\wtf=f$. Write $proj_i:M\rightarrow \Tor_i$ the projection and consider $(\ga,X)$ a good adapted field such that $X_{\cdot, k}$ is a leading component.  By hypothesis \textbf{S-2} on the map $S_k$, the vector field $(\ga^1_j,Y=f_{\ast}X)$ has its $k$-th component inside $\Delta_k$ provided that $proj_k(f^{-1}\ga^1_j)\not\in \mathcal{C}_k$.  It is no loss of generality to assume that for $1\leq i\leq e$
	\[
	proj_{i'}(p)\not\in\mathcal{C}_{i'}, \forall 1\leq i'\leq e\Rightarrow \norm{d_pS_i|\mathrm{C}\Delta_i}<\min_{i'\neq i} m(d_pS_h|\Delta_{i'})
	\]
	otherwise $\Delta_i$ can be enlarged while preserving its expanding character.
	
	Define $\ga_g=\{p\in \ga: proj_i(p)\not\in\mathcal{C}_{i}, 1\leq i\leq e\}$. Then $|\ga_g|\geq (1-\rho(r))|\ga|$, and by the discussion above if $f^{-1}\ga^1_j\subset \ga_g$ then either the $Y_k$ is its leading component and thus $(\ga^1_j,Y)$ is good, or the leading component of $Y$ is $Y_{k'}$, where $X_{k'}$ was expanded and hence was inside $\Delta_{k'}$. In any case $Y$ has a leading component $Y_i$ inside a cone $\Delta_i$ which implies that $(\ga^1_j,Y)$ is good. This completes the proof of the first part for the unperturbed case.  
	
	For a perturbation $\wtf$  we denote $\displaystyle{\pi_i: E_{\wtf}^c\rightarrow E_{i,\wtf}^c}$ the projection, and notice that $\pi_i(\wtf_{\ast}\pi_iX)$ is uniformly close to $\displaystyle{T_{\wtf}f_{\ast}T_{\wtf}^{-1}(\pi_iX)}$, thus we can reduce to the previous case. 
	
	To deal part (b) we first consider $\wtf=f$: we know that for all leading components $X_{\cdot,i}$ of $X$ there exists $p_0$ such that $X_{p_0}\in \mathrm{C}\Delta_i^{\varsigma_i}(p_0)$, $\varsigma_i=+,-$. Now for any other $p\in Im(\ga)$,
	\[
	d_pf(X_p)=d_pf(X_p-X_{m_0})+d_pf(X_{p_0})
	\] 
	and thus by Corollary \ref{variacionX} the first term of right hand side is very small, hence $d_pf(X_p)\approx 
	d_pf(X_{p_0})$. Consider $\mathcal{B}_{i,r}^{+}, \mathcal{B}_{i,r}^{-}$, the bands defined in condition \textbf{S-1}, and note that if $proj_i(\wtf^{-1}\ga^1_j)\subset \mathcal{B}_{i,r}^{\varsigma_i}$ then the $i$-th component of $(\ga^1_j,Y)$ is properly contained into $\Delta_i$, away from its boundary. Define 
	\[\ga_{bg}=\ga_g\cap\bigcup\{proj_i^{-1}\mathcal{B}_{i,r}^{\varsigma_i}:i/ X_i \text{ leading}\} 
	\] and note that for $r$ large 
	$|\ga_{bg}|\geq 0.99R|\ga|$. If $\wtf^{-1}\ga^1_j\subset \ga_{bg}$, then we have two possibilities:
	\begin{itemize}
		\item exists $i$ such that $X_i,Y_i$ are leading components of $X,Y$. Then by the previous argument together with \textbf{S-2} we have that 
		$Y_i$ is completely contained in a expanding cone, hence $(\ga^1_j,Y)$ is good, or
		\item the leading components of $Y$ come from non-leading components of $X$. In this case necessarily these non-leading components are in expanding cones, hence by the invariance of these cones we also deduce that $(\ga^1_j,Y)$ is good.
	\end{itemize}
	The result follows. Arguing as for the first part we can deal with the perturbative case. Compare Lemma 12 in \cite{NUHD}. 
\end{proof}

\esp

Consider now a good adapted field $(\ga,X)$ with $X\in E_{i,\wtf}^c$. We can estimate 
\[
\mathcal{E}\sum_{j\in G_1} \min_{\ga_j^1}J^u_{\wtf^{-1}}\gtrsim (2\pi-l(\mathcal{C}_{i,r}))\qquad \frac{1}{\mathcal{E}}\sum_{j\in B_1} \max_{\ga_j^1}J^u_{\wtf^{-1}} \lesssim l(\mathcal{C}_{i,r}).
\]
From this we deduce that for sufficiently large $r$ (and $\phi(r)$ small), it holds
\[
\sum_{j\in G_1} \min_{\ga_j^1}J^u_{\wtf^{-1}}>\sigma\sum_{j\in B_1} \max_{\ga_j^1}J^u_{\wtf^{-1}}
\]
where $\sigma$ is the natural number given in condition \textbf{S-1}. We continue to work with $r$ so the above holds.

Armed with the two previous lemmas now we will establish the following, which almost immediately implies Proposition \ref{fonda}.

\begin{proposition}\label{relacionesGB}
	If $r$ is sufficiently large and its corresponding $\phi(r)$ is sufficiently small, then the following  holds for every $(\ga,X)$ good adapted field for $\wtf$:  
	\begin{enumerate}
		\item $\sum_{j\in G_k} \min_{\ga_j^k}J^u_{\wtf^{-k}}\geq (1-10^{-23})\frac{\sigma}{\sigma+1}.$
		\item $\sum_{j\in G_k} \min_{\ga_j^k}J^u_{\wtf^{-k}}\geq  \sigma\sum_{j\in B_k} \max_{\ga_j^k}J^u_{\wtf^{-k}}.$
	\end{enumerate}
\end{proposition}

\begin{proof}
	We start noticing that the first part is consequence of the second together with Lemma \ref{numberg}. We argue by induction. The base case is just the hypothesis, so assume that we have established the claim for $k\geq 0$. By Lemma \ref{distortion}
	\begin{align*}
	&\sum_{j\in G_{k+1}} |\gamma_j^{k+1}| \min_{\gamma_j^{k+1}} J^u_{\wtf^{-k-1}}\geq 
	\frac{1}{\mathcal{E}}\sum_{j\in G_{k+1}}\int_{\gamma_j^{k+1}} J^u_{\wtf^{-k-1}} d(\ga^{k+1}_j)\\
	&\geq \frac{1}{\mathcal{E}}\sum_{j\in G_{k+1}}\sum_{t\in G_{k}}\int_{\gamma_j^{k+1}\cap \wtf(\gamma_t^{k})} J^u_{\wtf^{-k-1}} d(\ga^{k+1}_j)=\frac{1}{\mathcal{E}}\sum_{t\in G_{k}}\sum_{j\in G_{k+1}}\int_{\wtf^{-1}(\gamma_j^{k+1})\cap \gamma_t^{k}} J^u_{\wtf^{-k}} d(\ga^{k}_t)\\
	&\geq \frac{1}{\mathcal{E}}\sum_{t\in G_{k}}\min_{\ga_t^k} (J^u_{\wtf^{-k}})\sum_{j\in G_{k+1}}|\wtf^{-1}(\gamma_j^{k+1})\cap \gamma_t^{k}|\geq \frac{1}{\mathcal{E}}\sum_{t\in G_{k}}\min_{\ga_t^k}(J^u_{\wtf^{-k}})\cdot(1-\rho)\cdot|\gamma_t^{k}|\\
	\end{align*}
	where in the last line we have used Lemma \ref{tripforg}. Since the length of admissible curves is comparable (cf.\@ Corollary \ref{longcurvas}), we finally obtain
	\begin{equation}\label{bound1}
	\sum_{j\in G_{k+1}} \min_{\gamma_j^{k+1}} J^u_{\wtf^{-k-1}}	\geq \frac{1-\varpi}{\mathcal{E}(1+\varpi)}(1-\rho)\cdot\sum_{t\in G_{k}}\min_{\ga_t^k}(J^u_{\wtf^{-k}})
	\end{equation}
	On the other hand, and arguing in the same way
	\begin{align*}
	&\sum_{j\in B_{k+1}} |\gamma_j^{k+1}| \max_{\gamma_j^{k+1}} J^u_{\wtf^{-k-1}}\leq 
	\mathcal{E}\sum_{j\in B_{k+1}}\int_{\gamma_j^{k+1}} J^u_{\wtf^{-k-1}} d(\ga^{k+1}_j)\\
	&\leq \mathcal{E}\sum_{j\in B_{k+1}}\sum_{t\in G_{k}}\int_{\gamma_j^{k+1}\cap \wtf(\gamma_t^{k})} J^u_{\wtf^{-k-1}} d(\ga^{k+1}_j)+\mathcal{E}\sum_{j\in B_{k+1}}\sum_{t\in B_{k}}\int_{\gamma_j^{k+1}\cap \wtf(\gamma_t^{k})} J^u_{\wtf^{-k-1}} d(\ga^{k+1}_j)\\
	&+\mathcal{E}\int_{\wtf(\gamma_{N_{k}+1}^k))} J^u_{\wtf^{-k-1}} d(\wtf(\gamma_{N_{k}+1}^k))\\
	&=\mathcal{E}\sum_{t\in G_{k}}\sum_{j\in B_{k+1}}\int_{\wtf^{-1}(\gamma_j^{k+1})\cap \gamma_t^{k}} J^u_{\wtf^{-k}} d(\ga^{k}_t)+\mathcal{E}\sum_{t\in B_{k}}\sum_{j\in B_{k+1}}\int_{\wtf^{-1}(\gamma_j^{k+1})\cap \gamma_t^{k}} J^u_{\wtf^{-k}} d(\ga^{k}_t)\\
	&+\mathcal{E}\int_{\gamma_{N_{k}+1}^k} J^u_{\wtf^{-k}} d(\gamma_{N_{k}+1}^k)\\
	&\leq \mathcal{E}^22\rho\sum_{t\in G_{k}}\min_{\ga_t^k}(J^u_{\wtf^{-k}})\cdot|\gamma_t^{k}|+\mathcal{E}\cdot\frac{2\pi-0.99R}{2\pi}\sum_{t\in B_{k}}\max_{\ga_t^k}(J^u_{\wtf^{-k}})|\gamma_t^{k}|+\mathcal{E}\max_{\gamma_{N_{k}+1}^k}(J^u_{\wtf^{-k}})\cdot |\gamma_{N_{k}+1}^k|.
	\end{align*}
	Choose $\phi(r)$ sufficiently small so that for every $k\geq 0$,
	\[
	\max_{M}(J^u_{\wtf^{-k}})\leq \frac{(1-10^{-23})\sigma}{\sigma+1}\lambda^{-k/2}
	\]
	and use Lemma \ref{longcurvas} with the induction hypotheses to conclude
	\begin{equation}\label{bound2}
	\sum_{j\in B_{k+1}} \max_{\gamma_j^{k+1}} J^u_{\wtf^{-k-1}}\leq \frac{1+\varpi}{1-\varpi}\cdot \mathcal{E}\cdot \left( 2 \mathcal{E}\rho+\frac{2\pi-0.99R}{2\pi\sigma}+\lambda^{-k/2}\right)\sum_{t\in G_{k}}\min_{\ga_t^k}(J^u_{\wtf^{-k}})	
	\end{equation}
	and hence, putting together \eqref{bound1}, \eqref{bound2} and using that $\mathcal{E}\xrightarrow[\phi(r)\To 0]{} 1$,
	\begin{align*}
	\frac{\sum_{j\in G_{k+1}} \min_{\gamma_j^{k+1}} J^u_{\wtf^{-k-1}}}{\sum_{j\in B_{k+1}} \max_{\gamma_j^{k+1}} J^u_{\wtf^{-k-1}}}&\geq \Big(\frac{1-\varpi}{1+\varpi}\Big)^2\frac{1-\rho}{\mathcal{E}^2}\frac{2\pi}{4\pi \mathcal{E}\rho+\frac{2\pi-0.99R}{\sigma}+2\pi\lambda^{-k/2}}\\
	&>(1-10^{-23})\frac{2\pi(1-\rho)}{(2\pi-0.99R)}\cdot\sigma>\sigma
	\end{align*}
	if $\varpi(r),\rho(r)$ are sufficiently small.
\end{proof}

\esp

Finally we are ready to finish the proof of Proposition \ref{fonda}.

\begin{proof}\label{prueba}
	Define $\displaystyle{\beta:=\min_{1\leq i\leq e}\beta_i, \zeta:=\min_{1\leq i\leq e}\zeta_i}$ and let $(\ga,X)$ be a good adapted field for $\wtf$ with leading component $X_i\in \Delta_{i,\wtf}$. We compute using Lemma \ref{relacionEga} and Proposition \ref{relacionesGB}
	\begin{align*}
	&\sum_{j=0}^{N_k}
	\min_{\ga_j^k}J^u_{\wtf^{-k}} \cdot I(\ga_j^k, Y^k)=\sum_{j\in G_k}
	\min_{\ga_j^k}J^u_{\wtf^{-k}} \cdot I(\ga_j^k, Y^k)+\sum_{j\in B_k}\min_{\ga_j^k}J^u_{\wtf^{-k}} \cdot I(\ga_j^k, Y^k)\\
	&\geq \left(2-10^{-23})\pi\log \beta(r)+\frac{1}{\sigma}(1-10^{-23})\log \zeta_i(r)\right)\sum_{j\in G_k}
	\min_{\ga_j^k}J^u_{\wtf^{-k}} \cdot I(\ga_j^k, Y^k)\\
	&\geq \left(6\log \beta(r)+\frac{1}{\sigma}\log \zeta(r)\right)\frac{(1-10^{-23})\sigma}{\sigma+1}=\frac{(1-10^{-23})\sigma}{\sigma+1}\log(\beta(r)^6\zeta(r)^{1/\sigma})
	\end{align*}
	and the later quantity is positive, if $r$ large enough by the last part of \textbf{S-1}.	
\end{proof}
\section*{Acknowledgments}

The results here presented are based on previous joint work with Pierre Berger, and are deeply influenced by
several discussion that the author had with him during these times. I would like to thank Pierre for sharing his ideas with me.

Initial stages of this paper were prepared while I was visiting PUCV-Valparaiso; I would like to thank Carlos Vazquez for his encouragement and generosity in these moments. Also, I would like to thank Enrique Pujals and Jiagang Yang for encouragement and the interest deposited in this project.

Finally, I would like to express my sincere thanks to the referees who not only gave me many suggestions to improve the presentation and caught inaccuracies-plain errors, but also gave me ideas on how to improve the results appearing on previous versions.

\section*{Appendix}

Given a coupled family $\{f_r=A_r\times_{\varphi_r} S_r\}_{r}$ over an hyperbolic base,  it is desirable to have some conditions that will imply non-uniform hyperbolicity of the family instead of only positive exponents along the fiber direction. We discuss a possible approach here.

We assume that both $\{S_r\}_{r},\{S_r^{-1}\}_{r}$ satisfy \textbf{S-1},\textbf{S-2}, $A_r\in SL(2,\Z)$  hyperbolic, and define the following conditions.

\begin{enumerate}
	\item[\textbf{A-3}]
	\begin{itemize}
		\item $\displaystyle{\frac{\norm{dS^{-1}_r}^3\cdot \norm{dS_r}}{\lambda_r\norm{\varphi_r|E^s_{A_r}}}\xrightarrow[r\To+\oo]{}0,\frac{\norm{dS^{-1}_r}\cdot \norm{dS_r}\cdot\norm{\varphi_r|E_{A_r}^u}}{\lambda_r}\xrightarrow[r\To+\oo]{}0}$.\\
		
		\item There exists $q\in\mathbb{N}$ such that 
		\[
		\frac{\norm{dS_r}^{3q}\norm{d^2S_r^{-1}}^{{3q}}}{\lambda_r}\xrightarrow[r\To+\oo]{}0.
		\]	
		
	\end{itemize}
	\item[\textbf{A-4}]
	\begin{itemize}
		
		\item  $\displaystyle{\min_{1\leq i\leq e}\norm{P_i\circ dS_r^{-1}\circ\varphi_r|E^s_{A_r}}>0}$.
		
		\item $\displaystyle{\max_{1\leq\i\leq e}\frac{\norm{P_i\circ dS_r^{-1}\circ\varphi_r|E^u_{A_r}}+\norm{P_i\circ dS_r^{-1}}}{\lambda^2\norm{P_i\circ dS^{-1}\circ \varphi_r|E^s_{A_r}}}\xrightarrow[r\To+\oo]{}0}$.
		
		\item $\displaystyle{\frac{\min_{1\leq i\leq e}\norm{P_i\circ dS^{-1}\circ \varphi_r|E^s_{A_r}}}{\max_{1\leq i\leq e} \norm{P_i\circ dS^{-1}\circ \varphi_r|E^s_{A_r}}}\xrightarrow[r\To+\oo]{} 1.}$
	\end{itemize}
\end{enumerate}

\begin{definition}
	We say the the coupling in a family of skew-products $\{f_r=A_r\times_{\varphi_r} S_r\}_{r}$ is bi-adapted if it is adapted and moreover the family satisfies \textbf{A-3,A-4} above. In this case we also say that $\{f_r\}$ is a bi-adapted family.
\end{definition}

The lack is symmetry between these and \textbf{A-1,A-2} comes from the different form of $df$ and $df^{-1}$; compare \eqref{derivadaf} with \eqref{derivadafinv}.

It is direct (although somewhat tedious) to check that if $f_r$ verifies \textbf{A-1} to \textbf{A-4} then $df_r^{-1}$ verifies \textbf{A-1},\textbf{A-2}.  We can thus apply our Main Theorem to both $\{f_r\}_{r}, \{f_r^{-1}\}_{r}$ and deduce the following.

\begin{coroB}
	Assume that $\{f_r=A_r\times_{\varphi_r} S_r:M=\Toro^l\times\Toro^{2e}\rightarrow M\}_{r}$ is a bi-adapted family with $A_r\in SL(2,\Z)$ hyperbolic and the families $\{S_r\}_r,\{S_r^{-1}\}_r$ satisfy conditions \textbf{S-1,S-2}. Then there exists $r_0$ such that for every $r\geq r_0$ there exists $Q(r)>0$ satisfying for Lebesgue almost every $m\in M$
	\[
	v\in T_mM\setminus\{0\}\quad\Rightarrow\lim_{n\To\oo}\left|\frac{\log\norm{d_mf^n(v)}}{n}\right| >Q(r).
	\]
	In particular $f_r$ is NUH and has a physical measure.  The same holds for any $\wtf$ in a $\mathcal{C}^2$ neighborhood $\mathcal{U}_r$ of $f_r$.
\end{coroB}

The previous Corollary is given for completeness. In practice however, checking \textbf{A-4} could be difficult since it depends on the relation between $dS_r^{-1}$ and $d\varphi_r|E^u_A$, and this control may not be achievable. This is  why in the examples given in Section 3 we appeal to other arguments to deal with the inverse map.

\bibliographystyle{alpha}
\bibliography{biblio}

\begin{thebibliography}{CHHU17}

\bibitem[ABV00]{SRBMostexp}
J.F. Alves, C.~Bonatti, and M.~Viana.
\newblock {SRB measures for partially hyperbolic systems whose central
  direction is mostly expanding}.
\newblock {\em Invent. Math.}, 140(2):351--298, 2000.

\bibitem[Anz51]{anzai1951}
H.~Anzai.
\newblock Ergodic skew product transformations on the torus.
\newblock {\em Osaka Math. J.}, 3(1):83--99, 1951.

\bibitem[AV10]{ExtLyaExp}
A.~Avila and M.~Viana.
\newblock {Extremal Lyapunov exponents: an invariance principle and
  applications}.
\newblock {\em Invent. Math.}, 181(1):115--178, 2010.

\bibitem[BC91]{dynHenon}
M.~Benedicks and L.~Carleson.
\newblock {The dynamics of the H\'{e}non map}.
\newblock {\em Ann. of Math.}, 133:73--169, 1991.

\bibitem[BC14]{NUHD}
P.~Berger and P.~D. Carrasco.
\newblock Non-uniformly hyperbolic diffeomorphisms derived from the standard
  map.
\newblock {\em Communications in Mathematical Physics}, 329(1):239--262, 2014.

\bibitem[BDP02]{PartHypLya}
K.~Burns, D.~Dolgopyat, and Ya. Pesin.
\newblock {Partial Hyperbolicity, Lyapunov Exponents and Stable Ergodicity}.
\newblock {\em Journal of Statistical Physics}, 108(5):927--942, 2002.

\bibitem[BDPP08]{StaErgNeg}
K.~Burns, D.~Dolgopyat, Y.~Pesin, and M.~Pollicott.
\newblock Stable ergodicity for partially hyperbolic attractors with negative
  central exponents.
\newblock {\em Journal of Modern Dynamics}, 2(1):63--81, 2008.

\bibitem[BDV05]{Beyond}
C.~Bonatti, L.~D\'{i}az, and M.~Viana.
\newblock {\em Dynamics Beyond Uniform Hyperbolicity}, volume 102 of {\em
  Encyclopaedia of Mathematical Physics}.
\newblock Springer-Verlag, 2005.

\bibitem[Ber19]{Abundance}
P.~Berger.
\newblock {Abundance of non-uniformly hyperbolic H\'{e}non-like endomorphisms}.
\newblock {\em Ast\'{e}risque}, 410:53--176, 2019.

\bibitem[Bow08]{EquSta}
R.~Bowen.
\newblock {\em {Equilibrium States and the Ergodic Theory of Anosov
  Diffeomorphisms}}, volume 470 of {\em {Lect. Notes in Math.}}
\newblock Springer Verlag, 2008.

\bibitem[BR75]{BowenRuelle}
R.~Bowen and D.~Ruelle.
\newblock The ergodic theory of {A}xiom {A} flows.
\newblock {\em Invent. Math.}, 29:181--202, 1975.

\bibitem[BV00]{Bonatti2000}
C.~Bonatti and M.~Viana.
\newblock {SRB} measures for partially hyperbolic systems whose central
  direction is mostly contracting.
\newblock {\em Israel Journal of Mathematics}, 115(1):157--193, dec 2000.

\bibitem[BV05]{Bochi2005}
J.~Bochi and M.~Viana.
\newblock {The Lyapunov exponents of generic volume-preserving and symplectic
  maps}.
\newblock {\em Ann. of Math}, 161(3):1423--1485, may 2005.

\bibitem[BW99]{StErgSkew2}
K.~Burns and A.~Wilkinson.
\newblock {Stable ergodicity of skew products}.
\newblock {\em Ann. Sci. de le Ecole Norm. Sup.}, 32(6):859--889, 1999.

\bibitem[BXY17]{LyaRandom}
A.~Blumenthal, J.~Xue, and LS~Young.
\newblock Lyapunov exponents for random perturbations of some area-preserving
  maps including the standard map.
\newblock {\em Ann. of Math.}, 185:285--310, 2017.

\bibitem[BY93]{SRBHenon}
M.~Benedicks and L.S. Young.
\newblock {Sinai-Bowen-Ruelle measures for certain Henon maps}.
\newblock {\em Invent. Math.}, 112(1):541--576, 1993.

\bibitem[CHHU17]{survey3d}
P.~D. Carrasco, F.~R. Hertz, J.~R. Hertz, and R.~Ures.
\newblock Partial hyperbolicity in dimension three.
\newblock {\em Ergodic Theory and Dynamical Systems}, 2017.

\bibitem[Chi79]{chirikov1979}
B.~V. Chirikov.
\newblock A universal instability of many-dimensional oscillator systems.
\newblock {\em Physics Reports}, 52(5):263--379, may 1979.

\bibitem[CK15]{randomitera}
O.~Castejón and V.~Kaloshin.
\newblock Random iteration of maps on a cylinder and diffusive behavior.
\newblock {\em preprint at Arxiv}, 2015.

\bibitem[CP15]{CroPotPH}
S.~Crovisier and R.~Potrie.
\newblock Introduction to partially hyperbolic dynamics.
\newblock In {\em Lecture Notes for the School on Dynamical Systems, ICTP},
  2015.

\bibitem[CS08]{Chirikov2008}
B.~Chirikov and D.~Shepelyansky.
\newblock Chirikov standard map.
\newblock {\em Scholarpedia}, 3(3):3550, 2008.

\bibitem[Dol00]{Mostcont}
D.~Dolgopyat.
\newblock {On dynamics of mostly contracting diffeomorphisms}.
\newblock {\em Communications in Mathematical Physics}, 213(1):181--201, 2000.

\bibitem[{Dol}05]{Dolgoaverage}
D.~{Dolgopyat}.
\newblock {Averaging and invariant measures.}
\newblock {\em {Mosc. Math. J.}}, 5(3):537--576, 2005.

\bibitem[dSL18]{limitfastslow}
J.~de~Simoi and C.~Liverani.
\newblock Limit theorems for fast-slow partially hyperbolic systems.
\newblock {\em Invent. Math.}, 213:811--1016, 2018.

\bibitem[Dua94]{Plenty}
P.~Duarte.
\newblock {Plenty of elliptic islands for the standard family of area
  preserving maps}.
\newblock {\em Ann. Inst. Henri Poincare}, 11(4):359--409, 1994.

\bibitem[Fro72]{Froeschle}
C.~Froeschl\'e.
\newblock Numerical study of a four-dimensional mapping.
\newblock {\em Astronomy and Astrophysics}, 1972.

\bibitem[Gol01]{Gole2001}
C.~Gole.
\newblock {\em Symplectic Twist Maps: Global Variational Techniques (Advanced
  Series in Nonlinear Dynamics)}.
\newblock World Scientific Pub Co Inc, 2001.

\bibitem[Gor12]{stochasticsea}
A.~Gorodetski.
\newblock On stochastic sea of the standard map.
\newblock {\em Communications in Mathematical Physics}, 309(1):155--192, 2012.

\bibitem[HPS77]{HPS}
M.~Hirsch, C.~Pugh, and M.~Shub.
\newblock {\em {Invariant Manifolds}}, volume 583 of {\em {Lect. Notes in
  Math.}}
\newblock Springer Verlag, 1977.

\bibitem[Jak81]{Jakob}
M.~Jakobson.
\newblock {Absolutely continuous invariant measures for one-parameter families
  of one-dimensional maps}.
\newblock {\em Communications in Mathematical Physics}, 81(39-88), 1981.

\bibitem[KKM17]{Korepanov2017}
A.~Korepanov, Z.~Kosloff, and I.~Melbourne.
\newblock Averaging and rates of averaging for uniform families of
  deterministic fast-slow skew product systems.
\newblock {\em Studia Mathematica}, 238(1):59--89, 2017.

\bibitem[Kni96]{entKnill}
O.~Knill.
\newblock Topological entropy of standard type monotone twist maps.
\newblock {\em Transactions of the AMS}, 348(8), 1996.

\bibitem[Led84]{LedraExp}
F.~Ledrappier.
\newblock Quelques proprietes des exposants caracteristiques.
\newblock In P.~L. Hennequin, editor, {\em {\'E}cole d'{\'E}t{\'e} de
  Probabilit{\'e}s de Saint-Flour XII - 1982}, pages 305--396, Berlin,
  Heidelberg, 1984. Springer Berlin Heidelberg.

\bibitem[Mar16]{Marin2016}
K.~Marin.
\newblock $\mathcal{C}^r$-density of (non-uniform) hyperbolicity in partially
  hyperbolic symplectic diffeomorphisms.
\newblock {\em Commentarii Mathematici Helvetici}, 91(2):357--396, 2016.

\bibitem[Oba20]{Obata2018}
D.~Obata.
\newblock {On the Stable Ergodicity of Berger-Carrasco's example}.
\newblock {\em Ergodic Theory and Dynamical Systems}, 40(4):1008--1056, 2020.

\bibitem[{Ose}68]{Oseledets}
V.I. {Oseledets}.
\newblock {A multiplicative ergodic theorem. Lyapunov characteristic numbers
  for dynamical systems.}
\newblock {\em {Trans. Mosc. Math. Soc.}}, 19:197--231, 1968.

\bibitem[PC10]{Pesin2010}
Y.~Pesin and V.~Climenhaga.
\newblock Open problems in the theory of non-uniform hyperbolicity.
\newblock {\em Discrete and Continuous Dynamical Systems}, 27(2):589--607, feb
  2010.

\bibitem[Pes77]{LyaPesin}
Y.~Pesin.
\newblock {Characteristic Lyapunov exponents, and smooth ergodic theory}.
\newblock {\em Russian Math. Surveys}, 32(4):55--114, 1977.

\bibitem[Pes04]{PesinLect}
Y.~Pesin.
\newblock {\em {Lectures on Partial Hyperbolicity and Stable Ergodicity}}.
\newblock {Zurich Lectures in Advanced Mathematics}. European Mathematical
  Society, 2004.

\bibitem[Rob98]{Robinson1998}
C.~Robinson.
\newblock {\em Dynamical Systems: Stability, Symbolic Dynamics, and Chaos
  (Studies in Advanced Mathematics)}.
\newblock Studies in Advanced Mathematics. CRC Press, 2 edition, 1998.

\bibitem[RQ92]{Reversible}
J.A.G. Roberts and G.R.W. Quispel.
\newblock Chaos and time-reversal symmetry. order and chaos in reversible
  dynamical systems.
\newblock {\em Physics Reports}, 216(2-3):63--177, jul 1992.

\bibitem[Rue76]{SRBattractor}
D.~Ruelle.
\newblock A measure associated with axiom-a attractors.
\newblock {\em American Journal of Mathematics}, 98(3):619, 1976.

\bibitem[Shu86]{Shub}
M.~Shub.
\newblock {\em {Global Stability of Dynamical Systems}}.
\newblock Springer, 1986.

\bibitem[Via97]{Multinonhyp}
M.~Viana.
\newblock {Multidimensial non-hyperbolic attractors}.
\newblock {\em Publ. Math IHES}, 85(1):63--96, 1997.

\bibitem[WLL90]{Wood1990}
B.~P. Wood, A.~J. Lichtenberg, and M.~A. Lieberman.
\newblock Arnold diffusion in weakly coupled standard maps.
\newblock {\em Physical Review A}, 42(10):5885--5893, nov 1990.

\end{thebibliography}
\end{document}